\documentclass[letterpaper]{amsart}
\usepackage[utf8]{inputenc}

\usepackage{amsfonts}
\usepackage{amssymb}
\usepackage{tikz}
\usetikzlibrary{decorations.pathreplacing}
\usepackage{mathtools}
\usepackage{enumerate}
\usepackage{xcolor}
\usepackage{graphicx}
\usepackage{soul}
\usepackage{subcaption}
\usepackage[hide links]{hyperref}
\usepackage{comment}

\usepackage{concmath}
\usepackage{ccfonts}
\usepackage{eulervm}
\labelformat{subfigure}{\thefigure(\textsc{#1})}

\theoremstyle{plain}
\newtheorem{thm}{Theorem}
\newtheorem{prop}[thm]{Proposition}
\newtheorem{lem}[thm]{Lemma}
\newtheorem{cor}[thm]{Corollary}

\newtheorem{ques}{Question}

\theoremstyle{definition}

\newtheorem{ex}{Example}

\theoremstyle{remark}
\newtheorem*{rmk}{Remark}

\DeclareMathOperator{\dist}{dist}
\DeclareMathOperator{\ecc}{ecc}

\makeatletter
\def\smallunderbrace#1{\mathop{\vtop{\m@th\ialign{##\crcr
   $\hfil\displaystyle{#1}\hfil$\crcr
   \noalign{\kern3\p@\nointerlineskip}%
   \tiny\upbracefill\crcr\noalign{\kern3\p@}}}}\limits}
\makeatother

\title{On fixing and distinguishing numbers of trees}

\author[Buchanan]{Calum~Buchanan}
\address{Department of Mathematics, University of Denver, Denver, CO, USA}
\email{calum.buchanan@du.edu}

\author[Dankelmann]{Peter~Dankelmann}
\address{Department of Mathematics and Applied Mathematics, University of Johannesburg, Johannesburg, South Africa}
\email{pdankelmann@uj.ac.za}

\author[Harris]{Isabel~Harris}
\address{Department of Mathematics, Rice University, Houston, TX, USA}
\email{isabel.harris@rice.edu}

\author[Horn]{Paul~Horn}
\address{Department of Mathematics, University of Denver, Denver, CO, USA and Department of Mathematics and Applied Mathematics, University of Johannesburg, Johannesburg, South Africa}
\email{paul.horn@du.edu}

\author[Perry]{K.~E.~Perry}
\address{Mathematics, Soka University of America, Aliso Viejo, CA, USA}
\email{kperry@soka.edu}

\author[Rivett-Carnac]{Emily~Rivett-Carnac}
\address{Department of Mathematics and Applied Mathematics, University of Johannesburg, 
Johannesburg, South Africa}
\email{emilyjcarnac@gmail.com}

\date{\today}

\begin{document}

\dedicatory{This article is dedicated to the memory of Debra Boutin.}

\begin{abstract}

A graph $G$ is {\em $D$-distinguishable} if there is a labeling of its vertices with $D$ labels such that the only automorphism of $G$ which preserves the labeling is the identity. The {\em distinguishing number} of $G$ is the minimum value $D$ for which $G$ is $D$-distinguishable. The {\em fixing number} of $G$ is the minimum cardinality of a subset of the vertices of $G$ which is fixed pointwise only by the trivial automorphism. We prove that the fixing number of any $2$-distinguishable tree of order $n \geq 3$ is at most $4n/11$, or at most $(D-1)n / (D+1)$ for a $D$-distinguishable tree ($D \geq 3$). For every $D$ and $r$ at least $2$, we characterize the $D$-distinguishable trees with radius $r$ by constructing a universal tree $T_r^D$ which has the property that a tree $T$ of radius $r$ is $D$-distinguishable if and only if $T$ is a union of branches of $T_r^D$. We obtain a similar collection of universal trees for the property of having a constant paint cost spectrum, {\em i.e.}, the minimum size of the complement of a color class in a distinguishing $D$-coloring of $T$ is equal to the fixing number. Finally, we prove bounds on the distinguishing and fixing numbers of a tree in terms of the eccentricities of its vertices.

\end{abstract}

\maketitle


\section{Introduction}

A {\em distinguishing coloring} of a graph is a (not necessarily proper) vertex coloring whose color classes are preserved only by the trivial automorphism. In other words, each vertex of the graph can be identified uniquely by its color and graph properties. A graph $G$ is said to be {\em $D$-distinguishable} if there exists a distinguishing coloring of $G$ using at most $D$ colors and the {\em distinguishing number} of $G$, denoted $D(G)$, is the minimum value $D$ for which $G$ is $D$-distinguishable. Graph distinguishing was introduced independently by Albertson and Collins in~\cite{AC1996} and by Babai in~\cite{Ba1977} under the name {\em asymmetric coloring.}

A {\em fixing set} of $G$ is a subset of the vertex set $V(G)$ which is fixed pointwise only by the trivial automorphism. The {\em fixing number} of $G$, denoted $F(G)$, is the minimum cardinality of a fixing set. Fixing numbers were introduced independently by Harary in~\cite{harary96, H2001}, Boutin in~\cite{B2006} (under the name {\em determining number}), and Fijav\v{z} and Mohar in~\cite{fijavz2004rigidity} (under the name {\em rigidity index}). 

Distinguishing numbers and fixing numbers have been well-studied since their introduction, with the majority of work related to determining these parameters for specific classes of graphs. See~\cite{ACD2007, BCKLPR2020b,  b2009, Cha2006b, FNT2008, HIKST2017, WKT2007, STW2012}, for example. In general, however, not much work has been done considering the two parameters together, even though clear relationships exist. For example, if one colors each vertex in a fixing set of a graph $G$ a different color and colors the remaining vertices another color, this is a $(F(G)+1)$-distingushing coloring of $G$ and so it follows that $D(G) \leq F(G)+1$~\cite{AB2007}.
On the other hand, any union of all but one color classes in a distinguishing coloring of $G$ comprises a fixing set, so the fixing number of an $n$-vertex graph $G$ is at most $\frac{D(G) - 1}{D(G)} n$ by the pigeonhole principle.
We are interested in bounds of the latter type, and we define the \emph{fixing density} of $G$ to be $F(G) / n$.

It is also worth noting that, as both parameters measure how symmetric a graph is, the study of distinguishing numbers and fixing numbers also has nice applications to the graph isomorphism problem. In particular, breaking symmetry is a key tool in the design of algorithms for determining graph isomorphism~\cite{BC2011}. On the other hand, there are multiple classes of graphs for which the fixing number and/or the distinguishing number is known, but the automorphism group is not well understood~\cite{WKT2007, STW2012}.

The {\em paint cost} of a $2$-distinguishable graph, introduced by Boutin in 2008~\cite{B2008}, is the minimum size of a color class in a distinguishing $2$-coloring.
It turns out that, for a large number of graph families, all but a finite number of members are $2$-distinguishable. For example, hypercubes~\cite{BC2004}, Cartesian powers for most connected graphs~\cite{A2005, IK2006, KZ2007}, and Kneser graphs~\cite{AB2007} all have this property. 

This notion of paint cost was generalized in~\cite{B2023} to graphs $G$ of arbitrary distinguishing number $D$. For $t \geq D(G)$, let $\mathcal{D}^t(G)$ denote the set of distinguishing $t$-colorings of $G$. The {\em $t$-paint cost} $\rho^t(G)$ is 
\[
\min_{c \in \mathcal{D}^t(G)}{ \{ n - |C| : C \text{ a largest color class in } c \} }.
\]
It is not hard to see that $\rho^t$ is monotone decreasing with $t$, and in~\cite{B2023} it was shown that for $t \geq F(G)+1$, $\rho^t(G) = F(G)$. Mafunda, Merzel, Perry, and Varvak~\cite{MMPV2026} then extended the platform for studying generalized paint cost with the introduction of the \emph{paint cost spectrum}: $(D(G); \rho^{D(G)}(G), \rho^{D(G)+1}(G), \dots, \rho^{F(G)+1}(G))$. 
For instance, the paint cost spectrum of a cycle of length $6$ is $(2;3,2)$, as depicted in Figure~\ref{fig:paintcost}. Of particular interest are the graphs with constant paint cost spectrum, i.e. graphs $G$ such that $\rho^{D(G)}(G) = \rho^{F(G)+1}(G)$ and it was posted by Boutin in~\cite{B2023} to characterize this class.

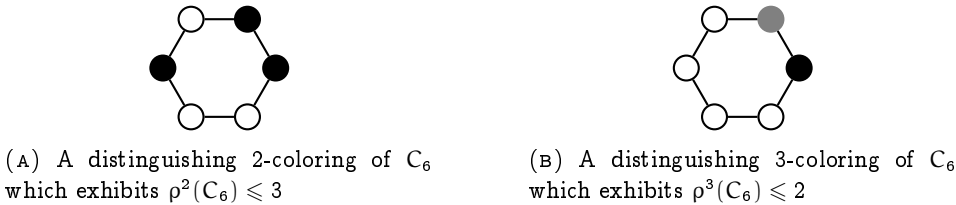
\begin{figure}
\begin{subfigure}{.45\textwidth}
    \centering
    \begin{tikzpicture}
        [thick, scale=0.75, minimum size=.3cm,
        blackvertex/.style={circle,thick,draw, fill=black!100},
        whitevertex/.style={circle, thick, draw},
        grayvertex/.style={circle,thick,draw,fill=black!50}]

        \foreach \i in {0,1,3}
        {
        \node[blackvertex] (\i) at (\i*360/6:1) {};
        }
        \foreach \i in {2,4,5}
        {
        \node[whitevertex] (\i) at (\i*360/6:1) {};
        }
        \draw (0) -- (1) -- (2) -- (3) -- (4) -- (5) -- (0);
    \end{tikzpicture}
    \caption{A distinguishing $2$-coloring of $C_6$ which exhibits $\rho^2(C_6) \leq 3$}
    \label{subfig:paintcost_2-coloring}
\end{subfigure}
\hfill
\begin{subfigure}{.45\textwidth}
    \centering
    \begin{tikzpicture}
        [thick, scale=0.75, minimum size=.3cm,
        blackvertex/.style={circle,thick,draw, fill=black!100},
        whitevertex/.style={circle, thick, draw},
        grayvertex/.style={circle,thick,draw=black!50,fill=black!50}]

        \node[blackvertex] (0) at (0:1) {};
        \node[grayvertex] (1) at (60:1) {};
        \foreach \i in {2,3,4,5}
        {
        \node[whitevertex] (\i) at (\i*360/6:1) {};
        }
        \draw (0) -- (1) -- (2) -- (3) -- (4) -- (5) -- (0);
    \end{tikzpicture}
    \caption{A distinguishing $3$-coloring of $C_6$ which exhibits $\rho^3(C_6) \leq 2$}
    \label{subfig:paintcost_3-coloring}
\end{subfigure}
	\caption{Colorings of the cycle of length $6$ which minimize the $2$- and $3$-paint cost, respectively.}
	\label{fig:paintcost}
\end{figure}

We will be interested in graphs, unlike the cycle, for which $\rho^t$ is constant as $t$ varies.
Equivalently, these are graphs with distinguishing number $D$ for which $\rho^{D}(G) = F(G)$.
The equivalence is due to the fact that $\rho^t$ is minimized when $t = F(G)+1$ by coloring a minimum fixing set rainbow.

Trees are a fundamental class of graphs that offer a rich setting to explore the connection between the distinguishing number and fixing number. For example, Cheng and, independently, Arvind and Devanur obtained $O(n \log n)$-time algorithms for finding the distinguishing number of trees~\cite{Che2006, AD2004}. In~\cite{Ty2004}, Tymoczko showed that for every tree $T$, $D(T) \leq \Delta(T)$, where $\Delta(T)$ is the maximum degree of a vertex in $T$, and recently, Alikhani and Soltani~\cite{AS2016a} characterized trees with distinguishing number two and radius at most three. 

We obtain sharp upper bounds on the fixing density of trees in terms of their distinguishing numbers. In particular, we show that a tree of order at least three that is $2$-distinguishable has fixing density at most $4/11$. More generally, when the distinguishing number of a tree is $D \geq 3$, its fixing density does not exceed $(D-1)/(D+1)$. 

The paper is organized as follows. Definitions and preliminary results pertaining to the distinguishing number and fixing number are introduced in Section \ref{sec: prelim results}. Bounds on the fixing density of $D$-distinguishable trees are obtained in Section \ref{sec:fixing densities}, while a characterization of all $D$-distinguishable trees with radius $r\geq 1$ is given in Section \ref{sec:universal trees}. Lastly, Section~\ref{sec: ecc sequences} deals with establishing bounds on the distinguishing number and fixing number of trees with a given eccentric sequence and we conclude the paper with some open questions in Section~\ref{sec: future work}.


\section{Definitions and preliminary results}
\label{sec: prelim results}

We establish the following definitions for our subsequent proofs. The order of a graph $G$ is denoted $|G|$. 
The degree of a vertex $v$ in $G$ is denoted $\deg_G(v)$, or simply $\deg(v)$ when the graph is clear from context.
When $G$ is connected, we denote the distance between vertices $v$ and $w$ by $\dist(v,w)$.
The {\em eccentricity} of $v$, denoted $\ecc(v)$, is $\max_{w \in V(G)} \dist(v,w)$.
The minimum eccentricity of a vertex in $G$ is called the {\em radius} of $G$, and a vertex obtaining this minimum is called a {\em central vertex}. The {\em diameter} of $G$ is the maximum ecccentricity of one of its vertices.

Throughout this paper, we are mainly concerned with trees: connected graphs without cycles.
It is well-known folklore that any given tree has either one or two central vertices, depending on whether its diameter is even or odd, respectively. 
As automorphisms are distance-preserving, it will often be useful for us to think of our trees as rooted at a central vertex.
The following result of Erwin and Harary on fixing sets of trees will also be of use.

\begin{thm}[\cite{EH2006}]\label{thm: leaf fixing set}
    For every tree $T$, there is a minimum cardinality fixing set consisting only of leaves of $T$.
\end{thm}

Let $T$ be a rooted tree with root $x$.
The {\em branches} $b_1, b_2, \dots, b_{\deg(x)}$ of $T$ are the maximal subtrees of $T$ which contain $x$ as a leaf.
A {\em branched subgraph} of $T$ is any nonempty union of its branches or the one-vertex subtree consisting only of $x$.

We refer to vertices of degree at least $3$ in a tree as {\em high-degree}. A {\em spider} is a tree with at most one high-degree vertex $x$. If $T$ is a spider with root vertex $x$, let $n_k(T)$ denote the number of branches of $T$ with $k$ edges. Then
\[
|T| = 1 + \sum k n_k(T). 
\]
Note that the fixing density of $T$, then, is
\[
\frac{\sum \max\{0,(n_k(T)-1)\}}{1 + \sum k n_k(T)}.
\]
We also observe the following.
\begin{lem} 
If $T$ is a $D$-distinguishable spider rooted at $x$, then $n_k(T) \leq D^{k}$ for all $k$.
\end{lem}
This follows as if there were more than $D^{k}$ paths of length $k$ appended to $x$, in any $D$-coloring, two paths would have the same coloring. 

Given the two observations above, we obtain the following lemma.

\begin{lem}\label{lem:all full or all empty}
Let $T$ be a $D$-distinguishable spider. If $n_k(T) \neq 0$ for some $k$, then there is a spider $T'$ whose fixing density is at least as large as $T$, so that $n_{k}(T') \in \{0,D^{k}\}$ with $n_{i}(T')=n_{i}(T)$ for $i \neq k$.  
\end{lem}

\begin{proof}
Fix $T$ and $k$ with $n_k = n_{k}(T) \neq 0$.  Then the fixing density of $T$ is 
\[
d(T) = \frac{n_k-1 + \sum_{i \neq k} \max\{0, n_i-1\}}{1 + k n_k + \sum_{i \neq k} i n_i}. 
\]

Then the derivative of the fixing density above with respect to $n_k$ is
\[
 \frac{(1 + kn_k + \sum_{i \neq k} i n_i) - k(n_{k}-1 + \sum_{i \neq k}\max\{0, n_i-1\})  }{(1 + kn_k + \sum_{i \neq k} i n_i)^2}. 
\]
Note that the sign of this derivative depends only on \[
1 + \sum_{i \neq k} in_i + k - k\sum_{i \neq k}\max\{0, n_i-1\}
\]
and is indepedent of $n_{k}$.  That is, if this sign is positive one gets a larger fixing density by increasing $n_k$ to $D^{k}$.  If the sign is negative one gets a larger fixing density by decreasing $n_k$ first to one (whence the formula breaks down) and then notes that decreasing $n_k$ from $1$ to $0$ clearly increases the fixing density as it decreases the denominator while leaving the numerator unchanged.  

As a final note we observe that if there is only one non-empty class (i.e. $n_i = 0$ for all $i \neq k$), then the derivative will be positive, so one can avoid obtaining a trivial tree.
\end{proof}

Lemma~\ref{lem:all full or all empty} allows us to conclude that if $T$ is a $D$-distinguishable spider, $D\geq 2$, with maximum fixing density, then $n_k(T) \in \{0,D^k\}$ for all $k$.


\section{Fixing densities of $D$-distinguishable trees} \label{sec:fixing densities}

In this section, we show that any $2$-distinguishable tree of order at least $3$ has fixing density at most $4/11$, and that any tree with distinguishing number $D \geq 3$ has fixing density at most $(D-1)/(D+1)$.
Further, both of these bounds are sharp. 

\begin{lem}\label{lem:4/11_spiders}
    If $T$ is a $2$-distinguishable tree of order at least $3$ with at most one high-degree vertex, then the fixing density of $T$ is at most $4/11$.
\end{lem}

\begin{proof}
Let $T$ be a $2$-distinguishable spider of order at least $3$; we may assume $T$ is not a path, as the fixing density of a path is $1/|T|$, which is strictly less than $4/11$ when $|T| \geq 3$.  

Write the fixing density 
\[
 \frac{\sum{\max\{0, n_k-1\}}}{1+\sum kn_k} = \frac{W+X}{Y+Z}  
\]
where \begin{align*} 
W &= \sum_{k=1}^{2} \max\{0,n_k-1\} &&X = \sum_{k\geq 3} \max\{0, n_k-1\}\\
Y&= 1+ \sum_{k=1}^{2} kn_k&& 
Z = \sum_{k\geq 3} kn_{k}. \end{align*}
Note that if $Z \neq 0$ 
\[\frac{X}{Z} = \frac{\sum_{k\geq 3} \max\{0, n_k-1\}}{ \sum_{k \geq 3} kn_{k}} < \frac{\sum_{k \geq 3}n_{k}}{3 \sum_{k \geq 3}n_k} = \frac{1}{3} < \frac{4}{11}.   \]

A simple real number inequality is that
\begin{equation} \label{ineq:fractions}
\frac{W+X}{Y+Z} \leq \max\left\{\frac{W}{Y}, \frac{X}{Z} \right\} 
\end{equation}
and hence it suffices to show $\frac{W}{Y} \leq  \frac{4}{11}$, as we have shown that either $\frac{X}{Z} < \frac{4}{11}$ or $X$ and $Z$ are both $0$. This follows by the fact that, per Lemma \ref{lem:all full or all empty}, $n_1 \in \{0,2\}$ while $n_2 \in \{0, 4\}$.  The ratio $\frac{4}{11}$ occurs exactly when $n_1 = 2$ and $n_2=4$. It is easy to verify that in the other cases, either the fixing density is less than $\frac{4}{11}$, or the tree has fewer than $3$ vertices.   
\end{proof} 

\begin{rmk}
    The unique $2$-distinguishable spider for which Lemma~\ref{lem:4/11_spiders} is sharp is the one depicted in Figure~\ref{fig:4/11}.
    A distinguishing $2$-coloring is shown in Subfigure~\ref{fig:distinguishing} and a minimum fixing set of four vertices is shown in Subfigure~\ref{fig:fixing}.
    
    To see that this is the only $2$-distinguishable spider of order at least $3$ and fixing density $4/11$, we need only note that the inequality~\eqref{ineq:fractions} is strict whenever $W/Y \neq X/Z$, that $X/Z < 4/11$ when the order $n$ of the spider is more than $11$, and that $W/Z < 4/11$ when $3 \leq n < 11$.
\end{rmk}

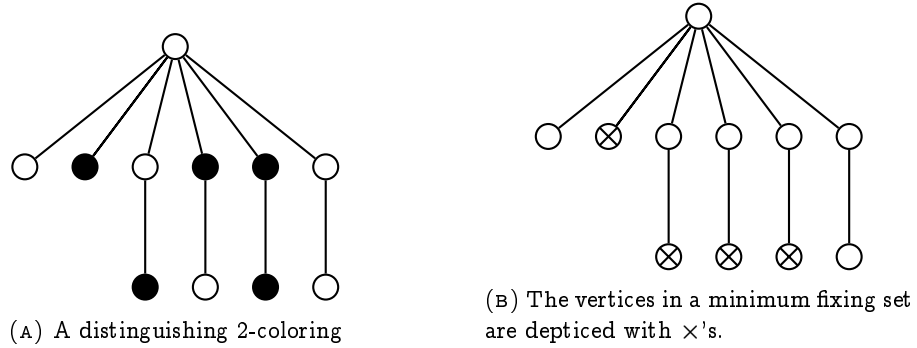
\begin{figure}
\begin{subfigure}{.45\textwidth}
    \centering
    \begin{tikzpicture}
        [scale=0.4, minimum size=.3cm,
        vertex/.style={circle,thick,draw}]
        
        \node[vertex] (1) at (6,8) [fill=white] {};
        \node[vertex] (2) at (1,4) [fill=white] {};
        \node[vertex] (3) at (3,4) [fill=black] {};
        \node[vertex] (4) at (5,4) [fill=white] {};
        \node[vertex] (5) at (5,0) [fill=black] {};
        \node[vertex] (6) at (7,4) [fill=black] {};
        \node[vertex] (7) at (7,0) [fill=white] {};
        \node[vertex] (8) at (9,4) [fill=black] {};
        \node[vertex] (9) at (9,0) [fill=black] {};
        \node[vertex] (10) at (11,4) [fill=white] {};
        \node[vertex] (11) at (11,0) [fill=white] {};
        
        \draw[thick] (2)--(1)--(3)--(1)--(4)--(5);
        \draw[thick] (7)--(6)--(1)--(8)--(9);
        \draw[thick] (1)--(10)--(11);
    \end{tikzpicture}
    \caption{A distinguishing $2$-coloring}
    \label{fig:distinguishing}
\end{subfigure}
\hfill
\begin{subfigure}{.45\textwidth}
    \centering
    \begin{tikzpicture}
        [scale=0.4, minimum size=.3cm,
        vertex/.style={circle,thick,draw},
        cross/.style={path picture={\draw[black] (path picture bounding box.south east) -- (path picture bounding box.north west) (path picture bounding box.south west) -- (path picture bounding box.north east);}}]
        \node[vertex] (1) at (6,8) [fill=white] {};
        \node[vertex] (2) at (1,4) [fill=white] {};
        \node[vertex,cross] (3) at (3,4) [fill=white] {};
        \node[vertex] (4) at (5,4) [fill=white] {};
        \node[vertex,cross] (5) at (5,0) [fill=white] {};
        \node[vertex] (6) at (7,4) [fill=white] {};
        \node[vertex,cross] (7) at (7,0) [fill=white] {};
        \node[vertex] (8) at (9,4) [fill=white] {};
        \node[vertex,cross] (9) at (9,0) [fill=white] {};
        \node[vertex] (10) at (11,4) [fill=white] {};
        \node[vertex] (11) at (11,0) [fill=white] {};
        
        \draw[thick] (2)--(1)--(3)--(1)--(4)--(5);
        \draw[thick] (7)--(6)--(1)--(8)--(9);
        \draw[thick] (1)--(10)--(11);
    \end{tikzpicture}
    \caption{The vertices in a minimum fixing set are depticed with $\boldsymbol{\times}$'s.}
    \label{fig:fixing}
\end{subfigure}
	\caption{A $2$-distinguishable tree with fixing density $4/11$.}
	\label{fig:4/11}
\end{figure}

We shall now prove that the upper bound of $4/11$ on the fixing density of a $2$-distinguishable spider of order at least $3$ holds for trees in general.
To do so, we define a {\em pendent path} in a graph to be a path with one endpoint of degree $1$, the other of degree at least $3$, and all interior vertices having degree $2$.
We say that a subset $F$ of vertices in a graph $G$ {\em fixes} a subgraph (or subset of vertices) $H$ if every automorphism of $G$ which fixes every vertex in $F$ also fixes every vertex in $H$.

\begin{thm}\label{thm:4/11}
	If $T$ is a $2$-distinguishable tree of order at least $3$, then the fixing density of $T$ is at most $4/11$, and this bound is sharp.
\end{thm}

\begin{proof}

	Let $U$ denote the set of vertices in $T$ which are the high-degree endpoints of at least one pendent path in $T$.
	For each $u \in U$, let $S_u$ denote the subgraph consisting of $u$ and all pendent paths in $T$ extending from $u$.
	We note that every leaf of $T$ is an endpoint of a pendent path (possibly of length $1$), and every pendent path in $T$ is contained in $S_u$ for some $u \in U$.
	Thus, while it is not always true that every vertex of $T$ is contained in some $S_u$, every leaf of $T$ is contained in some $S_u$. 

	We first claim that $S_u$ is $2$-distinguishable for every $u \in U$.
    If $\deg_{S_u}(u) \in \{1,2\}$, then $S_u$ is a path and we are done. On the other hand, suppose that $\deg_{S_u}(u) \geq 3$ and, for a contradiction, that $S_u$ is not $2$-distinguishable.
	Let $c$ be a distinguishing $2$-coloring of $T$, and let $c_u$ denote the restriction of $c$ to $S_u$.
    By supposition, there is a nontrivial automorphism $\phi_u$ of $S_u$ which preserves $c_u$.
    Since $u$ is the only vertex whose degree in $S_u$ is at least $3$, $\phi_u$ fixes $u$.
    This allows $\phi_u$ to be extended to a nontrivial automorphism $\phi$ of $T$ by defining $\phi(v) = v$ for all vertices $v$ in $T - S_u$.
    But now $\phi$ preserves the coloring $c$, which contradicts that $c$ was a distinguishing coloring of $T$.

    Theorem~\ref{thm: leaf fixing set} states that every tree has a minimum fixing set containing only leaves.
    For any subgraph $S_u$ of order at least $3$, $F(S_u) \leq 4 |S_u| / 11$ by Lemma~\ref{lem:4/11_spiders}.
    Since $4 |S_u| / 11 > 1$ when $|S_u| \geq 3$, we can find a fixing set $F_u$ for each such $S_u$ which contains at least one and at most $4/11$ of its vertices. Note that in the case where an $S_u$ is asymmetric, we will still select a leaf in $S_u$ to be the corresponding $F_u$.
    Further, every vertex in $F_u$ can be taken to be a leaf of $T$ (if $u$ is not a leaf in $S_u$, then the previously mentioned result assures us this is possible; otherwise $S_u$ is a path with endpoints $u$ and $w$, say, in which case $\{w\}$ is a fixing set for $S_u$).
	For the subgraphs $S_u$ of order $2$, we let $F_u$ be empty.
    
    We claim that the union $F$ of the $F_u$ for $u \in U$ fixes all of the leaves of $T$, and thus is a fixing set for $T$.
    Since the $S_u$ of order at least $3$ each have at least one vertex in $F$, and since $F_u \subseteq F$, no automorphism which preserves $F$ pointwise can permute leaves between or within these spiders.
    Thus, $F$ fixes all the leaves of $T$ which are not contained in a spider $S_u$ of order $2$.
    
    Any $S_u$ of order $2$ consists of a single leaf $w$ (a {\em lonesome dangler}) adjacent to $u$ in $T$.
    We claim that, if $P$ is a longest path in $T$ containing $u$, then the endpoints of $P$ lie in subgraphs $S_x, S_y$ for $x,y \in U$ with $|S_x|$ and $|S_y|$ at least $3$.
    As $S_x$ and $S_y$ are fixed by $F$, so shall be $P$, and thus $w$ as well.
    Clearly, the endpoints of $P$ are leaves, and thus contained in some spiders $S_x$ and $S_y$.
    For a contradiction, suppose that an endpoint of $P$ is a lonesome dangler $\ell$ contained in, say, $S_x$.
    Note that $\deg_T(x) \geq 3$ by definition of $U$, and that $\ell$ is the only leaf neighbor of $x$.
    Let $z$ be a neighbor of $x$ which is not contained in $P$.
    Then $P' := P - x\ell \cup xz$ is a path, and $z$ is not a leaf, so we can extend $P'$ to a longer path by adding a neighbor of $z$, contradicting that $P$ was a longest path containing $u$.
    This proves the claim, and thus we have that $F$ is a fixing set for $T$. 
    Further,
	\[|F| = \sum_{u \in U} |F_u| \leq \frac{4}{11}\sum_{u \in U} |S_u| \leq \frac{4}{11} |T|.\]

    Note that the argument above suggests an infinite family of sharpness examples: those trees $T$ for which $\{ S_u : u \in U\}$ is a partition of $V(T)$ and for which every $S_u$ is isomorphic to the spider depicted in Figure~\ref{fig:4/11}.
    One can obtain such a tree $T$ from a path $P$ on $k$ vertices by appending to each vertex a pendent spider isomorphic to the one in Figure~\ref{fig:4/11}.
    A distinguishing $2$-coloring of $P$ along with distinguishing $2$-colorings of the remaining vertices on each pendent spider provides a distinguishing $2$-coloring of $T$.
    Further, the fixing density of $T$ is precisely $4/11$.
\end{proof}

In a similar fashion, we now show that if $T$ is a tree with distinguishing number $D \geq 3$, then the fixing density of $T$ is at most $\frac{D-1}{D+1}$. We begin by showing this is true for a tree with exactly one-high degree vertex, before showing it is true for all $D$-distinguishable trees. 

\begin{lem}\label{lem:d>=3}
    If $T$ is a $D$-distinguishable tree, $D \geq 3$, with exactly one high-degree vertex, then the fixing density of $T$ is at most $\frac{D-1}{D+1}$.
\end{lem}

\begin{proof}

    Let $T$ be a $D$-distinguishable spider with one high degree vertex $x$. We note that the proof follows analogously to the proof of Lemma~\ref{lem:4/11_spiders}. Recall that by $n_k$, we denote the number of paths of $k$ vertices, not including $x$, appended to $x$. We can write the fixing density

    \[
 \frac{\sum{\max\{0, n_k-1\}}}{1+\sum kn_k} = \frac{W+X}{Y+Z}  
\]
where \begin{align*} 
W &=  \max\{0,n_1-1\} &&X = \sum_{k\geq 2} \max\{0, n_k-1\}\\
Y&= 1+  kn_1&& 
Z = \sum_{k\geq 2} kn_{k}. \end{align*}

Note that if $Z \neq 0$ and when $D \geq 3$
\[\frac{X}{Z} = \frac{\sum_{k\geq 2} \max\{0, n_k-1\}}{ \sum_{k \geq 2} kn_{k}} < \frac{\sum_{k \geq 2}n_{k}}{2 \sum_{k \geq 2}n_k} = \frac{1}{2} \leq \frac{D-1}{D+1}.   
\]

Again, we note that \[
\frac{W+X}{Y+Z} \leq \max\left\{\frac{W}{Y}, \frac{X}{Z} \right\} 
\] and observe that we have shown $ \frac XZ \leq \frac{D-1}{D+1}$ or $X$ and $Z$ are both 0. Thus, it suffices to show $\frac WY \leq \frac{D-1}{D+1}$. Of course, this inequality is tight precisely when $n_1 = D$ and is strict when $n_1 < D$.

\end{proof}

We now prove an extension of Theorem~\ref{thm:4/11} for trees with distinguishing number at least $3$.

\begin{thm}\label{thm:D-1/D+1}
    If $T$ is a tree with distinguishing number $D$ greater than or equal to 3, then the fixing density of $T$ is at most $\frac{D-1}{D+1}$, and this bound is sharp.
\end{thm}

\begin{proof}

The proof follows analogously to the proof of Theorem~\ref{thm:4/11} by observing that if $T$ is $D$-distinguishable, then each $S_u$ is $D$-distinguishable. By Lemmas~\ref{lem:4/11_spiders} and \ref{lem:d>=3}, we can fix any $S_u$ of order greater than or equal to 4 with $\frac{D-1}{D+1}|S_u|$ vertices, which we can select to be leaves. Let $F_u$ denote the fixing set of each respective $S_u$. As in the proof of Theorem~\ref{thm:4/11}, if any $S_u$ is asymmetric, we select a leaf in $S_u$ to add to the corresponding $F_u$ and observe that every $F_u$ contains at least one and at most $\frac{D-1}{D+1}|S_u|$ vertices. 

Any $S_u$ of order 3 is a path on 3 vertices with either 1 or 2 leaves in $T$. Let $\ell$ be a leaf in such an $S_u$ and add $\ell$ to $F_u$. Observe that $|F_u| = 1 = \frac{1}{3}|S_u| \leq \frac{D-1}{D+1}|S_u|$. 

Now let $F$ be the union of all the $F_u$. By the argument in the proof of Theorem~\ref{thm:4/11}, any $S_u$ of order 2 consists of a single leaf $w$ adjacent to a $u$ in $T$ and is fixed by $F$. Thus, $F$ is a fixing set for $T$ and we have that

\[|F| = \sum_{u \in U} |F_u| \leq \frac{D-1}{D+1}\sum_{u \in U} |S_u| \leq \frac{D-1}{D+1} |T|.\]

To see that the bound is sharp, let $k$ be a positive integer, and let $T_k$ be the tree obtained 
from a path on $k$ vertices by attaching $D$ leaves to each vertex. Then $T_k$ has $k(D+1)$
vertices, distinguishing number $D$ and fixing number $k(D-1)$. Hence  
$F(T_k) = \frac{D-1}{D+1}|T_k|$. 
\end{proof}

We conclude this section by noting that these bounds on the fixing density of a $D$-distinguishable tree are not necessarily true for $D$-distinguishable graphs in general. In particular, the graph $K_4$ with a pendent edge attached to each vertex is 2-distinguishable, but has fixing density greater than $4/11$, and any complete graph $K_n$ with $n \geq 3$ is $n$-distinguishable, but has fixing density greater than $\frac{n-1}{n+1}$. However, it would be of interest to identify classes of graphs for which sharper bounds can be obtained.
We return to this question in Section~\ref{sec: future work}.


\section{Universal $D$-distinguishable trees}\label{sec:universal trees}

A graph $G$ is \emph{universal} for a class of graphs $\mathcal{C}$ if $G$ contains every graph in $\mathcal{C}$ as a subgraph, up to isomorphism. The study of universal graphs was introduced by Rado in 1964~\cite{rado1964universal} and has grown significantly over the past 60 years.

Alikhani and Soltani gave a characterization of all trees with radius at most $3$ and distinguishing number $2$, as well as a necessary condition for trees with distinguishing number $2$ and radius more than $3$~\cite{AS2016a}. In this section, we provide a characterization of all $D$-distinguishable trees with radius $r$ ($D \geq 2$, $r \geq 1$) by establishing the existence of a universal tree. 

Recall that the $t$-paint cost of a graph $G$, $\rho^t(G)$, for $t \geq D(G)$, is the minimum size of the complement of a largest color class over all $t$-distinguishing colorings of $G$. In~\cite{B2023}, Boutin posed the open question to characterize all graphs for which $\rho^D(G) = F(G)$. Towards answering this question, we also provide a characterization of all $D$-distinguishable trees $T$ with radius $r$ ($D \geq 2$, $r \geq 1$) such that $\rho^D(T) = F(T)$, by establishing the existence of a universal tree. 

Our main contribution then is the construction, for each radius $r \geq 1$ and $D \geq 2,$ of a rooted tree $T_r^D$ which is universal. It is universal in the sense that it contains every $D$-distinguishable tree of radius $r$ as a branched subgraph, and furthermore every branched subgraph of $T_r^D$ of radius $r \geq 2$ is, indeed, $D$-distinguishable. (Recall that a branched subgraph is a union of branches or the singleton subgraph consisting of the root.)

\begin{thm}\label{thm:universal d-distinguishing trees}    
    Let $D$ and $r$ be positive integers with $D \geq 2$.
    There exists a rooted tree $T_r^D$ which contains all $D$-distinguishable trees of radius at most $r$ as branched subgraphs.
    Further, every branched subgraph of $T_r^D$ having radius $r \geq 2$ is $D$-distinguishable.
\end{thm}

\begin{proof}

    Let $D \geq 2$ be fixed.
    We begin by constructing $T_r^D$, the universal tree of radius $r \geq 1$ and distinguishing number $D$, by recursion on $r$. We shall let $x$, or $x_r^D$, denote the root of $T_r^D$.
    Along the way, we prove that $T_r^D$ contains all $D$-distinguishable trees of radius at most $r$ as branched subgraphs.

    For $r = 1$, we let $T_1^D = K_{1,D}$, where the root $x$ is the vertex of degree $D$.
    Note that the only $D$-distinguishable trees with radius $1$ are stars $K_{1,t}$ with $1 \leq t \leq D$, and these are all branched subgraphs of $K_{1,D}$.
    
    For $r \geq 2$, we construct the universal tree $T_r^D$ recursively from the universal tree $T_{r-1}^{D}$.
    We begin by taking all possible branched subgraphs of $T_{r-1}^{D}$, including the trivial ones consisting only of $x_{r-1}^D$. Copies of these subgraphs will be the components of $T_r^D - x$. We form $T_r^D$ by attaching as many copies of each branch as possible to a new root vertex $x$ by adding edges between $x$ and the copies of $x_{r-1}^D$, while maintaining $D$-distinguishability. (See Example~\ref{tdr_ex} and Figure~\ref{figure:D-distiguishable Universal} for the construction of $T_2^D$.) 

    We claim this construction yields a $D$-distinguishable tree $T_r^D$ of radius $r$
    which contains every $D$-distinguishable tree of radius at most $r$ as a branched subgraph.
    The claim is clearly true for $r = 1$.
    Supposing that it is true for $r - 1$, where $r \geq 2$, let $S$ be a $D$-distinguishable tree of radius at most $r$.
    If $y$ is a central vertex of $S$, then each component of $S - y$ is a $D$-distinguishable tree of radius at most $r - 1$, and is thus a branched subgraph of $T_{r-1}^D$ by induction.
    Noting that $T_r^D$ is constructed by taking as many copies of each branched subgraph of $T_{r-1}^D$ as possible while retaining $D$-distinguishability, were $S$ not to be a branched subgraph of $T_r^D$ (with $y$ mapped to the root $x_r^D$), $S$ would not be $D$-distinguishable, a contradiction.

    We now prove that, if $T$ is a branched subgraph of $T_r^D$ with radius $r \geq 2$, then $T$ is $D$-distinguishable.
    Since $T$ has radius $r$, it contains distinct branches $b_i, b_j$ such that the eccentricity of $x$ in $b_i$ is $r$, and the eccentricity of $x$ in $b_j$ is either $r$ or $r-1$, depending on whether the center of $T$ contains one or two vertices.
    In the former case, $x$ is the unique central vertex of $T$, so it is fixed by every automorphism.
    It then suffices to give each branch a distinguishing $D$-coloring to distinguish $T$, which is provided by the natural distinguishing $D$-coloring of $T_r^D$.
    In the latter case, $x$ is one of two central vertices in $T$, and the other central vertex $y$ is the neighbor of $x$ in $b_i$.
    We again assign each branch a $D$-distinguishing coloring, and we check on the color of $y$.
    By assigning $x$ any other color, we ensure that $x$ and $y$ are distinguished.
    Thus, we have a distinguishing coloring of $T$ using at most $D$ colors, which completes the proof.
\end{proof}

To better illustrate the construction of the universal tree for a given radius $r \geq 1$ and distinguishing number $D \geq 2$, we provide the following example of the universal trees of radius 2, depicted in Figure~\ref{figure:D-distiguishable Universal}.

\begin{ex}\label{tdr_ex}

Let $T_r^D$ denote the universal tree with root $x_r^D$, radius $r \geq 1$, and distinguishing number $D \geq 2$.

For radius $r = 2$,
we construct the universal tree $T_2^D$ from $T_1^D$. Recall that, for $r = 1$, the universal tree is the star $K_{1,D}$ with root having degree $D$. The universal tree $T_2^D$ is formed by taking all possible branched subgraphs of $K_{1,D}$
and attaching as many copies of each branched subgraph as possible (while maintaining $D$-distinguishability) to a new vertex $x_2^D$ via an edge from the copy of $x_1^D$.

In particular, $T_1^D$ has two types of branched subgraphs:
singletons ($x_1^D$) and star graphs $K_{1,i}$, $1 \leq i \leq D$, rooted at $x_1^D$.
Each of these types may be repeated in the following way:
\begin{enumerate}
    \item For singletons: $D$ different colors are possible, so $D$ singletons are possible without repeating a color on one of the vertices, which would permit an automorphism.
    
    \item For $K_{1,i}$ with $1\leq i\leq D$: 
    There are two ways $K_{1,i}$ can be distinguished: either every vertex is colored differently, or one non-root vertex shares a color with the root $x_1^D$. If every vertex has a different color, then there are $D$ choices for the color of the root and $\binom{D-1}{i}$ choices for the other $i$ vertices. If one vertex shares a color with the root, then there are $D$ choices for the shared color and $\binom{D-1}{i-1}$ choices for the remaining $i-1$ vertices. Thus, there are $D\binom{D-1}{i-1}+D\binom{D-1}{i} = D \binom{D}{i}$ distinct colorings of the $K_{1,i}$ possible without allowing an automorphism between stars.
\end{enumerate}

\begin{figure}[!h]
    \centering
    \begin{subfigure}{\textwidth}
        \centering
        \begin{tikzpicture}
        [every node/.style={circle, thick, draw=black!100, fill=none, minimum size=.3cm}]
    
        \node (r) at (0,2) {};
        \node (1) at (-2,1) {};
        \node (2) at (0,1) {};
        \node (3) at (2,1) {};
        \node (4) at (0,0) {};
        \node (5) at (1.25,0) {};
        \node (6) at (2.75,0) {};
    
        \foreach \i in {1,2,3}
        {
        \draw[thick] (r) -- (\i);
        }
        \draw[thick] (2) -- (4);
        \draw[thick] (5) -- (3) -- (6);
        \draw[thick] (3) -- (2,.5);
    
        \draw (2,.7) arc (270:310:.5);
        \draw (2,.7) arc (270:230:.5);
        \node[rectangle, draw=none, fill=none] (ibounds) at (3.4,.75) {\small $i, 2 \leq i \leq D$};
    
        \node[rectangle, draw=none, fill=none] (dots) at (2,0) {$\dots$};
    
        \node[rectangle, draw=none, fill=none] (brace1) at (-2,.6) {$\smallunderbrace{\hspace{.25cm}}$};
    
        \node[draw=none,fill=none] (exp1) at (-2,.25) {\small $D$};
    
        \node[rectangle, draw=none, fill=none] (brace2) at (0,-.4) {$\smallunderbrace{\hspace{.25cm}}$};
    
        \node[rectangle, draw=none,fill=none] (exp2) at (0,-.75) {\small $D^2$};
    
        \node[rectangle, draw=none, fill=none] (brace3) at (2,-.4) {$\underbrace{\hspace{2cm}}$};
    
        \node[rectangle, draw=none,fill=none] (exp3) at (2,-.8) {\small 
        $D \binom{D}{i}$};
    
        \end{tikzpicture}
        \caption{$T_2^D$}
        \label{subfig:D-distinguishable Universal a}
    \end{subfigure}
    
    \bigskip
    
    \begin{subfigure}{\textwidth}
        \centering
        \begin{tikzpicture}
        [every node/.style={circle, semithick, draw=black!100, fill=none, minimum size=3pt, inner sep=0pt}, scale=.25]
    
        \node (r) at (19,0) {};
        \foreach \i in {0,3,6,9}
        {
        \node (a\i) at (\i, -6) {};
        \draw (a\i) -- (r);
        }

        \node (r) at (19,0) {};
        \foreach \i in {1,4,7,10}
        {
        \node[draw=black!50, fill=black!50] (a\i) at (\i, -6) {};
        \draw (a\i) -- (r);
        }

        \node (r) at (19,0) {};
        \foreach \i in {2,5,8,11}
        {
        \node[fill=black!100] (a\i) at (\i, -6) {};
        \draw (a\i) -- (r);
        }
    
        \foreach \i in {12,15,18}
        {
        \node (a\i) at (2*\i - 11, -6) {};
        \draw (a\i) -- (r);
        }

        \foreach \i in {13,16,19}
        {
        \node[draw=black!50, fill=black!50] (a\i) at (2*\i - 11, -6) {};
        \draw (a\i) -- (r);
        }

        \foreach \i in {14,17,20}
        {
        \node[fill=black] (a\i) at (2*\i - 11, -6) {};
        \draw (a\i) -- (r);
        }

        \foreach \i in {12,13,14}
        {
        \node (b\i1) at (2*\i - 11.5, -10) {};
        \node[draw=black!50, fill=black!50] (b\i2) at (2*\i - 10.5, -10) {};
        }

        \foreach \i in {15,16,17}
        {
        \node (b\i1) at (2*\i - 11.5, -10) {};
        \node[fill=black] (b\i2) at (2*\i - 10.5, -10) {};
        }

        \foreach \i in {18,19,20}
        {
        \node[draw=black!50, fill=black!50] (b\i1) at (2*\i - 11.5, -10) {};
        \node[fill=black] (b\i2) at (2*\i - 10.5, -10) {};
        }

        \foreach \i in {12,...,20}
        {
        \draw (b\i1) -- (a\i) -- (b\i2);
        }

        \node (a21) at (3*21 - 31, -6) {};
        \node[fill=black!50, draw=black!50] (a22) at (3*22 - 31, -6) {};
        \node[fill=black] (a23) at (3*23 - 31, -6) {};

        \foreach \i in {21,22,23}
        {
        \draw (a\i) -- (r);
        \node (b\i1) at (3*\i - 32, -10) {};
        \node[draw=black!50, fill=black!50] (b\i2) at (3*\i - 31, -10) {};
        \node[fill=black] (b\i3) at (3*\i - 30, -10) {};
        \draw (b\i1) -- (a\i) -- (b\i2);
        \draw (b\i3) -- (a\i);
        }
        
        \foreach \i in {3,4,5}
        {
        \node (b\i) at (\i, -10) {};
        \draw (b\i) -- (a\i);
        }
        
        \foreach \i in {6,7,8}
        {
        \node[draw=black!50, fill=black!50] (b\i) at (\i, -10) {};
        \draw (b\i) -- (a\i);
        }

        \foreach \i in {9,10,11}
        {
        \node[fill=black] (b\i) at (\i, -10) {};
        \draw (b\i) -- (a\i);
        }
        
        \end{tikzpicture}
        \caption{$T_2^3$}
        \label{figure:D-distinguishable Universal b}
    \end{subfigure}
    \caption{Universal D-distinguishable Trees with $r = 2$. 
    }
    \label{figure:D-distiguishable Universal}
\end{figure}
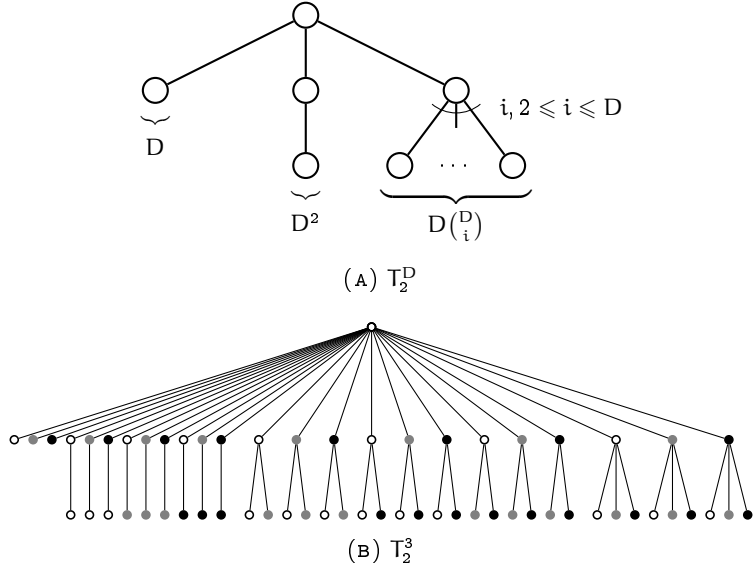

\end{ex}

We note here that, as the proof of Theorem~\ref{thm:universal d-distinguishing trees} demonstrates how to construct a finite universal tree $T$ of radius $r$ for any $r \geq 1$, it follows that:

\begin{cor}\label{cor:bounded deg}
    If $S$ is a $D$-distinguishable tree of radius less than or equal to $r$ ($r \geq 1$, $D \geq 2$), then $S$ has bounded maximum degree.
\end{cor}

In particular, we can observe that when $r = 2$, any $D$-distinguishable tree has maximum degree less than or equal to \[
D + D^2 + D \sum_{i=2}^{D} \binom{D}{i}  = D2^{D}. \]

We note also that, while it is true by Theorem~\ref{thm:universal d-distinguishing trees} that $T_r^D$ contains every $D$-distinguishable tree of radius at most $r$ as a branched subgraph and that any branched subgraph of $T_r^D$ of radius $r$ is $D$-distinguishable, it is not necessarily the case that if $S$ is a branched subgraph of $T_r^D$ of radius strictly less than $r$ that $D(S) = D$. For example, when $D \geq 2$, any universal tree with radius $r > 1$ and designated root $x$ will have a branched subgraph that is $K_{1,D+1}$, which is not $D$-distinguishable. 

We now establish the existence of a universal tree $U_r^D$ for all $D$-distinguishable trees $T$ with radius $r \geq 1$, $D \geq 2$, such that $\rho^D(T) = F(T)$.

\begin{thm}\label{thm:kp-good universal d-distinguishing trees}
    Let $D$ and $r$ be positive integers with $D \geq 2$.
    There exists a rooted tree $U_r^D$ which contains every $D$-distinguishable tree $T$ having radius at most $r$ and $\rho^D(T) = F(T)$ as a branched subgraph.
    Further, for $r \geq 2$, every branched subgraph $T$ of $U_r^D$ having radius $r$ is $D$-distinguishable and is such that $\rho^D(T) = F(T)$.
\end{thm}
\begin{proof}
    We construct $U_r^D$, the universal tree of radius $r \geq 1$ and distinguishing number $D \geq 2$ with the property that $\rho^D(S) = F(S)$, by recursion.

For $r = 1$, we define the universal tree $U_1^D$ to be the star $K_{1,D}$ with designated root $x$ adjacent to $D$ leaves. When $D = 2$, the only trees with radius 1 are $K_2 = K_{1,1}$ and $K_{1,2}$, which are both branched subgraphs of $K_{1,2}$. Additionally, $\rho^2(K_2)=1=F(K_2)$ and $\rho^2(K_{1,2})=1=F(K_{1,2})$. When $D \geq 3$, the only $D$-distinguishable trees with radius $1$ are star graphs $K_{1,D}$ and $\rho^D(K_{1,D})=D-1=F(K_{1,D})$. Thus, $K_{1,D}$ is a universal tree of radius $1$ with distinguishing number $D$. 

For $r \geq 2$, the universal tree $U_r^D$ is formed by taking all possible branched subgraphs of $U_{r-1}^D$ and attaching as many copies of each branched subgraph as possible (while maintaining $D$-distinguishability and the property that $\rho^D(U_r^D) = F(U_r^D)$) to a new vertex $x_r^D$ via an edge from the copy of $x_{r-1}^D$.
The proof that this construction is universal follows similarly to that of Theorem~\ref{thm:universal d-distinguishing trees}.
\end{proof}

As above, to better illustrate the construction, we provide the following example of the universal tree of radius 2, $U_2^D$, with distinguishing number $D \geq 2$ and the property that $\rho^D(U_2^D) = F(U_2^D)$.

\begin{ex}\label{kptdr_ex}
    Let $U_r^D$ denote the universal tree of radius $r \geq 1$ with distinguishing number $D \geq 2$ and the property that $\rho^D(U_r^D) = F(U_r^D)$ and let $x_r^D$ denote the root of $U_r^D$. 

    For radius $r = 2$ and a given distinguishing number $D$, we can construct the universal tree $U_2^D$ from the universal tree $U_1^D$. Recall that for $r = 1$, the universal tree is the star $K_{1,D}$ with root having degree $D$. The universal tree $U_2^D$ is formed by taking all possible branched subgraphs of $K_{1,D}$ and attaching as many copies of each branched subgraph as possible (while maintaining $D$-distinguishability and the property that $\rho^D(U_r^D) = F(U_r^D)$) to a new vertex $x_2^D$ via an edge from the copy of $x_1^D$. In a $D$-distinguishing coloring, we refer to vertices in the largest color class as receiving a neutral color. 

    In particular, $U_1^D$ has two types of branched subgraphs: singletons ($x_1^D$)
    and star graphs $K_{1,i}$, $1 \leq i \leq D$, rooted at $x_1^D$.
    Each of these types may be repeated in the following way:
    \begin{enumerate}
        \item For singletons: To maintain $D$-distinguishability, we can have at most $D$ singletons with $D-1$ vertices receiving a color that is not neutral. Additionally, the fixing number of $D$ singletons is $D-1$, so it follows that $D$ singletons are permitted. 
        
        \item For $K_{1,i}$ stars, $2\leq i\leq D$: $F(K_{1,i})=i-1$. To maintain $D$-distinguishability with non-neutral colors on no more than $i-1$ vertices for each $K_{1,i}$, we may color $x_1^D$ and one of the leaves with a neutral color. There are then $\binom{D-1}{i-1}$ choices for the remaining $i-1$ vertices.

        Note that $F(K_{1,1}) = 1$, and there are two different ways that a $K_{1,1}$ component of $U_2^D - x_2^D$ can be colored with at most one non-neutral color while maintaining $D$-distinguishability of $U_{2}^D$: either one vertex receives a non-neutral color, or both vertices are neutral.
        Thus, there are $2(D-1) + 1 = 2D - 1$ copies of $K_{1,1}$ possible.
    \end{enumerate}
\end{ex}
    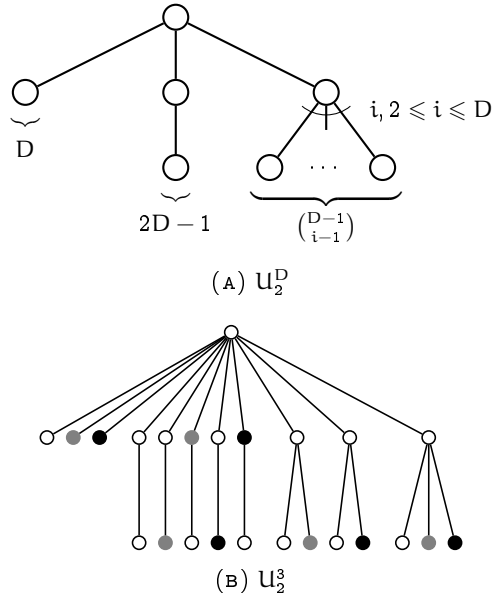
\begin{figure}[!h]
    \centering
    \begin{subfigure}{\textwidth}
    \centering
    \begin{tikzpicture}
    [every node/.style={circle, thick, draw=black!100, fill=none, minimum size=.3cm}]

    \node (r) at (0,2) {};
    \node (1) at (-2,1) {};
    \node (2) at (0,1) {};
    \node (3) at (2,1) {};
    \node (4) at (0,0) {};
    \node (5) at (1.25,0) {};
    \node (6) at (2.75,0) {};

    \foreach \i in {1,2,3}
    {
    \draw[thick] (r) -- (\i);
    }
    \draw[thick] (2) -- (4);
    \draw[thick] (5) -- (3) -- (6);
    \draw[thick] (3) -- (2,.5);

    \draw (2,.7) arc (270:310:.5);
    \draw (2,.7) arc (270:230:.5);
    \node[rectangle, draw=none, fill=none] (ibounds) at (3.4,.75) {\small $i, 2 \leq i \leq D$};

    \node[rectangle, draw=none, fill=none] (dots) at (2,0) {$\dots$};

    \node[rectangle, draw=none, fill=none] (brace1) at (-2,.6) {$\smallunderbrace{\hspace{.25cm}}$};

    \node[draw=none,fill=none] (exp1) at (-2,.25) {\small $D$};

    \node[rectangle, draw=none, fill=none] (brace2) at (0,-.4) {$\smallunderbrace{\hspace{.25cm}}$};

    \node[rectangle, draw=none,fill=none] (exp2) at (0,-.75) {\small $2D - 1$};

    \node[rectangle, draw=none, fill=none] (brace3) at (2,-.4) {$\underbrace{\hspace{2cm}}$};

    \node[rectangle, draw=none,fill=none] (exp3) at (2,-.8) {\small $\binom{D-1}{i-1}$};

    \end{tikzpicture}
    \caption{$U_2^D$}
    \label{subfig:KP good D}
    \end{subfigure}

    \bigskip
    
    \begin{subfigure}{\textwidth}
    \centering
    \begin{tikzpicture}
        [semithick, every node/.style={circle, semithick, draw=black!100, fill=none, minimum size=5pt, inner sep=0pt}, scale=.35]
    
        \node (r) at (7,0) {};

        \node (a0) at (0,-4) {};
        \node[draw=black!50, fill=black!50] (a1) at (1,-4) {};
        \node[fill=black] (a2) at (2,-4) {};

        \foreach \i in {0,...,2}
        {
        \draw (a\i) -- (r);
        }

        \node (a3) at (3.5,-4) {};
        \node (a4) at (4.5,-4) {};
        \node[draw=black!50, fill=black!50] (a5) at (5.5,-4) {};
        \node (a6) at (6.5,-4) {};
        \node[fill=black] (a7) at (7.5,-4) {};

        \node (b3) at (3.5,-8) {};
        \node[draw=black!50, fill=black!50] (b4) at (4.5,-8) {};
        \node (b5) at (5.5,-8) {};
        \node[fill=black] (b6) at (6.5,-8) {};
        \node (b7) at (7.5,-8) {};

        \foreach \i in {3,...,7}
        {
        \draw (a\i) -- (r);
        \draw (b\i) -- (a\i);
        }

        \node[draw=black!50, fill=black!50] (b82) at (10,-8) {};
        \node[fill=black] (b92) at (12,-8) {};
        
        \foreach \i in {8,9}
        {
        \node (a\i) at (2*\i - 6.5, -4) {};
        \draw (a\i) -- (r);
        \node (b\i1) at (2*\i - 7, -8) {};
        \draw (b\i1) -- (a\i) -- (b\i2);
        }
    
        \foreach \i in {10}
        {
        \node (a\i) at (3*\i - 15.5, -4) {};
        \draw (a\i) -- (r);
        \node (b\i1) at (3*\i - 16.5, -8) {};
        \node[draw=black!50, fill=black!50] (b\i2) at (3*\i - 15.5, -8) {};
        \node[fill=black] (b\i3) at (3*\i - 14.5, -8) {};
        \draw (b\i1) -- (a\i) -- (b\i2);
        \draw (b\i3) -- (a\i);
        }
        
        \end{tikzpicture}
        \caption{$U_2^3$}
        \label{sufig:U_2^3}
    \end{subfigure}
    \caption{Universal $D$-distinguishable tree for the property $\rho^{D(T)}(T) = F(T)$
    }
    \label{figure:KP-Good D}
    \end{figure}

We note here that the proof of Theorem~\ref{thm:kp-good universal d-distinguishing trees} demonstrates how to construct a finite universal tree $U$ of radius $r$, $r \geq 1$, with the property that $\rho^D(U) = F(U)$, and it follows that:

\begin{cor}
    If $S$ is a $D$-distinguishable tree of radius less than or equal to $r$ with the property that $\rho^D(S) = F(S)$, $r \geq 1$, $D \geq 2$, then $S$ has bounded maximum degree.
\end{cor}

Finally, for the same reasoning given for $T_r^D$, we note that not every branched subgraph $T$ of $U_r^D$ is $D$-distinguishable when the radius of $T$ is strictly less than $r$.
As an example, one can again consider $K_{1,D+1}$ as a branched subgraph of $U_r^D$ ($r \geq 2$), which has distinguishing number $D + 1$.
On the other hand, it is true that $\rho^{D+1}(K_{1,D+1}) = F(K_{1,D+1})$, and it remains open whether or not a branched subgraph $T$ of $U_r^D$ can be formed such that $\rho^{D(T)}(T) \neq F(T)$.


\section{Eccentric sequences} \label{sec: ecc sequences}

The {\em eccentric sequence} of a connected graph $G$ is the nondecreasing sequence of 
the eccentricities of its 
vertices. If $G$ has diameter $d$ and radius $r$, and $m_i$ denotes the number of 
vertices of eccentricity $i$, then we write the eccentric sequence
as $\bigl(r^{(m_{r})}, (r+1)^{(m_{r+1})}, (r+2)^{(m_{r+2})}, \ldots, d^{(m_d)}\bigr)$, and denote
it by $X(G)$. 

Lesniak characterized the eccentric sequences of trees as follows (see~\cite{lesniak1975eccentric}).

\begin{thm}[\cite{lesniak1975eccentric}]   \label{theo:lesniak}
Let $r,d \in \mathbb{N}$ and $X=\bigl(r^{(m_r)}, (r+1)^{(m_{r+1})},\ldots,d^{(m_d)}\bigr)$
be a sequence of nonnegative integers. Then $X$ is the eccentric sequence of some tree if
and only if the following conditions hold: \\
(i) $m_i \geq 2$ for $i=r+1, r+2,\ldots,d$, and \\
(ii) either $d=2r$ and $m_r=1$, or $d=2r-1$ and $m_r=2$.  
\end{thm}

We now give sharp upper bounds on the distinguishing number and the fixing number of a tree in
terms of its eccentric sequence. We define two families of exceptional trees. 
For integers $k_1,k_2,\ldots \in \mathbb{N}$, the spider $S_{k_1,k_2,\ldots}$ is the tree of order 
$1+k_1+k_2+\ldots$ obtained from paths of orders $k_1, k_2,\ldots$ by appending these to a single 
vertex. For $k \in \mathbb{N}$, the spider $S_{k,1,1,\ldots,1}$ is called a broom of diameter 
$k+1$. By the end of the broom we mean the leaf of $S_{k,1,1,\ldots,1}$ which is part of the 
path of length $k$.   
Let ${\mathcal D}$ be the family of trees containing all paths, stars, and all spiders of the form
$S_{r,r,k}$ where $1 \leq k < r$. (See Figure~\ref{subfiga:exceptions_D}.)
Let ${\mathcal F}$ be the family of trees containing 
the trees in ${\mathcal D}$ and, in addition, the tree $S_{r,r,1,1,\ldots,1}$ and the trees 
obtained from a path of length $2r$ and a broom of diameter at most $r$ by identifying the central vertex of 
the path with the end of the broom.
(See Figure~\ref{subfigb:exceptions_F-D}.)

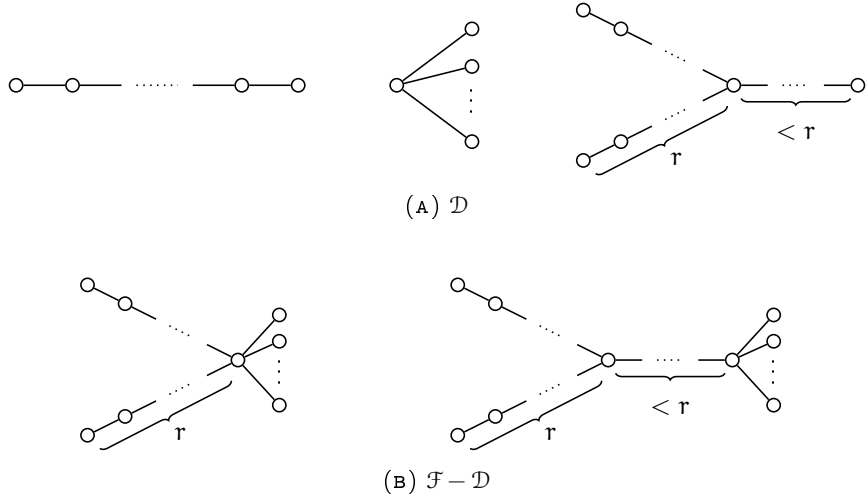
\begin{figure}
    \centering
    \begin{subfigure}{.9\textwidth}
        \centering
        \begin{tikzpicture}
        [semithick, every node/.style={circle, draw=black!100, inner sep=0pt, minimum size=5pt}, scale=.75, baseline=-28]
        
        \node (0) at (0,0) {};
        \node (1) at (1,0) {};
        \node[draw=none] (2) at (2,0) {};
        \node[draw=none] (3) at (3,0) {};
        \node (4) at (4,0) {};
        \node (5) at (5,0) {};
        
        \draw (0) -- (1) -- (2);
        \draw[dotted] (2) -- (3);
        \draw (3) -- (4) -- (5);
        
        \end{tikzpicture}
        \hfill
        \begin{tikzpicture}
        [semithick, every node/.style={circle, draw=black!100, inner sep=0pt, minimum size=5pt}, baseline=-28pt]
        
        \node (0) at (0,0) {};
        \node (1) at (1,.75) {};
        \node (2) at (1,.25) {};
        \node[draw=none] (dots) at (1,-.1) {$\vdots$};
        \node (3) at (1,-.75) {};
        
        \draw (1) -- (0) -- (2);
        \draw (0) -- (3);
        
        \end{tikzpicture}
        \hfill
        \begin{tikzpicture}
        [semithick, every node/.style={circle, draw=black!100, inner sep=0pt, minimum size=5pt}, baseline=-28pt]
        
        \node (0) at (0,0) {};
        \node (nw0) at (-2,1) {};
        \node (nw1) at (-1.5,.75) {};
        \node[draw=none] (nw2) at (-1,.5) {};
        \node[draw=none] (nw3) at (-.5,.25) {};
        \node (sw0) at (-2,-1) {};
        \node (sw1) at (-1.5,-.75) {};
        \node[draw=none] (sw2) at (-1,-.5) {};
        \node[draw=none] (sw3) at (-.5,-.25) {};

        \foreach \i in {n,s}
        {
        \draw (\i w0) -- (\i w1) -- (\i w2);
        \draw[dotted] (\i w2) -- (\i w3);
        \draw (0) -- (\i w3);
        }

        \draw[decoration={brace,mirror,raise=4pt},decorate]
        (-1.9,-1.05) -- node[draw=none, label={[label distance=7pt]-60:$r$}] {} (-.15,-.15);

        \node[draw=none] (e1) at (.55,0) {};
        \node[draw=none] (e2) at (1.1,0) {};
        \node (e3) at (1.65,0) {};

        \draw (0) -- (e1);
        \draw[dotted] (e1) -- (e2);
        \draw (e2) -- (e3);

        \draw[decoration={brace,mirror,raise=5pt},decorate]
        (0.east) -- node[draw=none, label={[label distance=7pt]below:$<r$}] {} (e3.west);
        
        \end{tikzpicture}
        \caption{$\mathcal{D}$}
        \label{subfiga:exceptions_D}
    \end{subfigure}
    \\[2em]
    \begin{subfigure}{.75\textwidth}
        \centering
        \begin{tikzpicture}
        [semithick, every node/.style={circle, draw=black!100, inner sep=0pt, minimum size=5pt}]
        
        \node (0) at (0,0) {};
        \node (nw0) at (-2,1) {};
        \node (nw1) at (-1.5,.75) {};
        \node[draw=none] (nw2) at (-1,.5) {};
        \node[draw=none] (nw3) at (-.5,.25) {};
        \node (sw0) at (-2,-1) {};
        \node (sw1) at (-1.5,-.75) {};
        \node[draw=none] (sw2) at (-1,-.5) {};
        \node[draw=none] (sw3) at (-.5,-.25) {};

        \foreach \i in {n,s}
        {
        \draw (\i w0) -- (\i w1) -- (\i w2);
        \draw[dotted] (\i w2) -- (\i w3);
        \draw (0) -- (\i w3);
        }

        \draw[decoration={brace,mirror,raise=4pt},decorate]
        (-1.9,-1.05) -- node[draw=none, label={[label distance=7pt]-60:$r$}] {} (-.15,-.15);

        \node (e1) at (.55,.6) {};
        \node (e2) at (.55,.25) {};
        \node (e3) at (.55,-.6) {};
        \node[draw=none] (dots) at (.55,-.05) {$\vdots$};

        \draw (e2) -- (0) -- (e1);
        \draw (0) -- (e3);
        
        \end{tikzpicture}
        \hfill
        \begin{tikzpicture}
        [semithick, every node/.style={circle, draw=black!100, inner sep=0pt, minimum size=5pt}]
        
        \node (0) at (0,0) {};
        \node (nw0) at (-2,1) {};
        \node (nw1) at (-1.5,.75) {};
        \node[draw=none] (nw2) at (-1,.5) {};
        \node[draw=none] (nw3) at (-.5,.25) {};
        \node (sw0) at (-2,-1) {};
        \node (sw1) at (-1.5,-.75) {};
        \node[draw=none] (sw2) at (-1,-.5) {};
        \node[draw=none] (sw3) at (-.5,-.25) {};

        \foreach \i in {n,s}
        {
        \draw (\i w0) -- (\i w1) -- (\i w2);
        \draw[dotted] (\i w2) -- (\i w3);
        \draw (0) -- (\i w3);
        }

        \draw[decoration={brace,mirror,raise=4pt},decorate]
        (-1.9,-1.05) -- node[draw=none, label={[label distance=7pt]-60:$r$}] {} (-.15,-.15);

        \node[draw=none] (e1) at (.55,0) {};
        \node[draw=none] (e2) at (1.1,0) {};
        \node (e3) at (1.65,0) {};

        \draw (0) -- (e1);
        \draw[dotted] (e1) -- (e2);
        \draw (e2) -- (e3);

        \draw[decoration={brace,mirror,raise=5pt},decorate]
        (0.east) -- node[draw=none, label={[label distance=7pt]below:$<r$}] {} (e3.west);

        \node (d1) at (2.2,.6) {};
        \node (d2) at (2.2,.25) {};
        \node (d3) at (2.2,-.6) {};
        \node[draw=none] (dots) at (2.2,-.05) {$\vdots$};

        \draw (d2) -- (e3) -- (d1);
        \draw (e3) -- (d3);
        
        \end{tikzpicture}
        \caption{$\mathcal{F} - \mathcal{D}$}
        \label{subfigb:exceptions_F-D}
    \end{subfigure}
    \caption{The exceptions to the bounds in Theorem~\ref{theo:fixing-distinguishing-vs-eccentric-sequence}}
    \label{fig:exceptions}
\end{figure}

\begin{thm} \label{theo:fixing-distinguishing-vs-eccentric-sequence}
Let $T$ be a tree of diameter $d$ and radius $r$, and let 
$(r^{(m_r)}, (r+1)^{(m_{r+1})},\ldots,d^{(m_d)})$ be its eccentric sequence.  Then
\begin{equation} \label{eq:distinguishing-vs-eccentricity-1} 
D(T) \leq  \max\{m_d-1, \max_{r \leq i <d}(m_i-2) \}, 
\end{equation}
unless $T \in {\mathcal D}$, and  
\begin{equation} \label{eq:fixing-vs-eccentricity-1} 
F(T) \leq  m_d-2 + \sum_{i=r}^{d-1} \max\{m_i-3,0\}  \}, 
\end{equation}
unless $T \in {\mathcal F}$.
\end{thm}

\begin{proof}
Let $T$ be a tree with eccentric sequence $(r^{(m_r)}, (r+1)^{(m_{r+1})},\ldots,d^{(m_d)})$. 
It is easy to verify that  
\eqref{eq:distinguishing-vs-eccentricity-1} holds if $T \in {\mathcal F} - {\mathcal D}$. 
Hence we may assume that $T \notin {\mathcal F}$.  

For $i = r, r+1,\ldots,d$ define $V^{(i)}$ to be the set of vertices of $T$ of 
eccentricity $i$, Then $|V^{(i)}| = m_i$, and each automorphism of $T$ fixes 
$V^{(i)}$ setwise. 
Let $T^{(d)}$ be the Steiner tree of the set $V^{(d)}$, i.e., the smallest subtree of 
$T$ containing all vertices of $V^{(d)}$. It is easy to see that the leaves of $T^{(d)}$ 
are exactly the vertices in $V^{(d)}$.  By Theorem \ref{thm: leaf fixing set} we 
thus obtain the following.  
\\[1mm] 
{\sc Claim 1:} If an automorphism $\phi$ of $T$ fixes $V^{(d)}$ pointwise, 
then $\phi$ fixes $T^{(d)}$. \\[1mm]
It is well-known that for every longest path $P$ of $T$ and every vertex $x$ of $T$, the eccentricity 
of $x$ is the distance to an end of $P$ farthest from $x$. Since $T^{(d)}$ contains all longest paths
of $T$, we thus have the following claim. \\[1mm]
{\sc Claim 2:} For every vertex $x$ of $T^{(d)}$, we have 
${\rm ecc}_{T}(x) = {\rm ecc}_{T^{(d)}}(x)$. \\[1mm]
Let $W$ be the set of vertices of $T$ not in $T^{(d)}$. 
For $i=r, r+1,\ldots,d$ define $W^{(i)} = W \cap V^{(i)}$. 
Since every automorphism of $T$ fixes $V^{(i)}$ and $V(T^{(d)})$, 
setwise, it fixes also $W^{(i)}$ setwise. 
Note that $W^{(r)}=\emptyset$ since the central vertices of a tree are on every longest path, and so
$V^{(r)}$ is contained in $T^{(d)}$. \\[1mm]

We first prove \eqref{eq:distinguishing-vs-eccentricity-1}.
Let $M(T)$ be the right hand side of \eqref{eq:distinguishing-vs-eccentricity-1}.
Suppose to the contrary that $T$ is a tree for which $D(T) > M(T)$. 
 
Define a coloring $c_W$ of the vertices of $W$ by assigning, for each $i \in \{r+1,r+2,\ldots,d-1\}$, distinct colors from $\{1,2,\ldots,|W^{(i)}| \}$ to the vertices of $W^{(i)}$. Since each 
automorphism of $T$ fixes $W$ and $V^{(i)}$, and thus $W^{(i)}$ setwise, and since the vertices
of $W^{(i)}$ receive different colors, it follows that every color-preserving automorphism of $T$
fixes $W$ pointwise. The number of different colors in $c_W$ equals 
$\max_{r+1 \leq i \leq d-1} |W^{(i)}|$. If $i\in \{r+1,r+2,\ldots d-1\}$, then $T^{(d)}$ has at 
least two vertices of eccentricity $i$ by Theorem \ref{theo:lesniak},  
and by Claim 2 their eccentricity in $T$ is also $i$. 
Hence $|W^{(i)}| \leq m_i-2$, and so the number of different colors of $c_W$ is at most
$\max_{r+1 \leq i \leq d-1}(m_i-2)$. 
If $c_d$ is a distinguishing coloring of $T^{(d)}$ with colors $1,2,\ldots,k$, say, then 
combining $c_d$ and $c_W$ yields a distinguishing coloring of $T$. Hence we have 
\begin{equation} \label{eq:eccentric-distinguishing-proof-1}
D(T) \leq \max\big\{ D(T^{(d)}), \max_{r \leq i \leq d-1}(m_i-2) \big\}.
\end{equation}   
We consider two cases: $M(T) \geq 2$ and $M(T) \leq 1$. In the former, we prove that $D(T^{(d)}) \leq M(T)$. In the latter, we derive the contradiction that $T \in \mathcal{D}$. \\[1mm]
{\sc Case 1:} $M(T) \geq 2$. \\
To prove \eqref{eq:distinguishing-vs-eccentricity-1} it suffices to show that 
$D(T^{(d)}) \leq M(T)$. 
Since $T$ is not a star, we have $d\geq 3$, and so 
there exist two vertices $u,v \in V^{(d)}$ which do not have a common neighbor. Let $u'$ and $v'$
be their respective neighbours. Let $c_d$ be a coloring of $T^{(d)}$ with 
$c_d(u)=c_d(v)=1$, $c_d(u')=1$, $c_d(v')=2$, and the vertices in $V^{(d)}-\{u,v\}$ receive 
distinct colors from $\{2,3,\ldots,m_d-1\}$. The remaining vertices of $T^{(d)}$ 
receive color $1$. It is easy to see that every color-preserving automorphism fixes $V^{(d)}$ and thus, 
by Claim 1, all of $T^{(d)}$. The number of colors of $c_d$ is $\max\{m_d-1, 2\}$, which
is not more than $M(T)$. 
This proves that $D(T^{(d)}) \leq M(T)$, and so \eqref{eq:distinguishing-vs-eccentricity-1} follows
in this case. 
\\[1mm]
{\sc Case 2:} $M(T) \leq 1$. \\
Since $m_i \geq 2$ for all 
$i\in \{r+1, r+2,\ldots,d\}$ by Theorem \ref{theo:lesniak}, we have $M(T)=1$, $m_d=2$ 
and $m_i \in \{2, 3\}$ for $i \in \{r+1, r+2,\ldots,d-1\}$. 
If the automorphism group of $T$ is trivial, then the theorem holds, so we may assume for the sake of contradiction that $T$ has a nontrivial automorphism $\phi$. 

Let $P$ be the unique path between the two vertices in $V^{(d)}$, so $p=T^{(d)}$. Then 
$P$ is the unique longest path of $T$, and as above $\phi$ fixes $V(P)$, and $V(T)-V(P)$ setwise.
Since $V(T)-V(P)$ contains no two vertices of the same eccentricity -- otherwise,
$m_i \geq 4$ for some $i$, a contradiction -- it follows that $\phi$ fixes each vertex
that does not belong to $P$. Since $\phi$ is nontrivial, $\phi$ transposes the
two vertices in $V^{(d)}$, i.e. the ends of $P$. Since $T$ is not a path, at least one vertex of $P$ is 
adjacent to a vertex not on $P$, and thus fixed by $\phi$. This implies 
that $P$ has even length and only the unique center vertex $w$ of $P$
has a neighbor outside $P$. For each $i$ with $i \in \{r+1, r+2, \ldots,d-1\}$ there
is at most one vertex of eccentricity $i$ not on $P$. Hence $T-V(P)$ is a path, and one
of its ends is adjacent to the center vertex of $P$. If $k:= |T|-|P|$, then $k <r$, 
otherwise we would have $m_d >2$. We thus obtain the contradiction $T=S_{r,r,k}$, 
and \eqref{eq:distinguishing-vs-eccentricity-1} follows in this case. \\[1mm]

We now prove \eqref{eq:fixing-vs-eccentricity-1}.  
Let $i \in \{r,r+1,\ldots,d-1\}$. 
If $W^{(i)} \neq \emptyset$, then define  $U^{(i)}$ to be the set obtained from $W^{(i)}$ by
removing a single vertex. If $W^{(i)} = \emptyset$, then define $U^{(i)} :=\emptyset$.
Let $U = \bigcup_{i=r+1}^{d-1} U^{(i)}$. An automorphism of $T$ fixes each $W^{(i)}$ setwise, 
and so an automorphism of $T$ that fixes $\bigcup_{i=r+1}^{d-1} U^{(i)}$ pointwise necessarily 
fixes $W$ pointwise.  
Let $i\in \{r+1, r+2,\ldots,d-1\}$. Since by Theorem \ref{theo:lesniak} the tree $T^{(d)}$ contains 
at least two vertices of eccentricity $i$, and these vertices have eccentricity $i$ in $T$ by 
Claim 2, we have $|W^{(i)}| \leq m_i-2$ and so $|U^{(i)}| \leq \max\{ m_i-3, 0\}$.
If $S$ is any fixing set of $T^{(d)}$, then  $S \cup \bigcup_{i=r+1}^{d-1} U^{(i)}$ is a fixing set 
of $T$. Since a set obtained from $V^{(d)}$ by removing one vertex clearly fixes $V^{(d)}$ and 
thus, by Claim 1, $T^{(d)}$, we have $F(T^{(d)}) \leq m_d-1$. In total we obtain 
\begin{equation}
F(T)  \leq  F(T^{(d)}) + \big|\bigcup_{i=r+1}^{d-1} U^{(i)}\big|  
      \leq  m_d-1 + \sum_{i=r+1}^{d-1} \max\{ m_i-3, 0\}     \label{eq:eccentric-fixing-proof-1}
\end{equation}
If the inequality \eqref{eq:eccentric-fixing-proof-1} is strict, then 
\eqref{eq:fixing-vs-eccentricity-1} holds. It thus suffices to prove that equality in 
 \eqref{eq:eccentric-fixing-proof-1} holds only if $T \in \mathcal{F}$. 
Assume that $T$ is a tree for which  \eqref{eq:eccentric-fixing-proof-1} holds with equality. Then
\begin{equation} \label{eq:eccentric-fixing-proof-2}
F(T^{(d)}) = m_d-1 \ \textrm{and} \ |U^{(i)}| 
                = \max\{m_i-3, 0\} \ \textrm{for $i \in \{r+1,r+2,\ldots,d-1\}$}.
\end{equation}
{\sc Claim 3:} $T^{(d)}=S_{r,r,\ldots,r}$, where   $r$ is repeated 
$m_d$ times. \\
To prove Claim 3 it suffices to show that 
$d_T(u,v)=d$ for all $u,v \in V^{(d)}$, $u \neq v$.
Suppose to the contrary that there exist $u,v \in V^{(d)}$ with $d(u,v) \neq d$.
Since there exists at least one pair of vertices at distance $d$ in $V^{(d)}$, there
exist vertices $u,v,w$ with $d(u,v) \neq d(u,w)$. We claim that the set $V^{(d)} - \{v,w\}$ 
is a fixing set of $T^{(d)}$. Indeed, if $\phi$ is an automorphism of $T^{(d)}$ that fixes
$V^{(d)} - \{v,w\}$ pointwise then $\phi$ cannot swap $v$ and $w$ since $\phi$ 
preserves the distances to $u$, so $\phi$ fixes also $v$ and $w$ pointwise. By Claim 1,
$\phi$ fixes $T^{(d)}$. Hence $F(T^{(d)}) \leq m_d-2$, a contradiction 
to \eqref{eq:eccentric-fixing-proof-2}. This proves Claim 3. \\[1mm]
Since $T^{(d)} = S_{r,r,\ldots,r}$, it follows that for $i=r+1, r+2,\ldots,d-1$, 
$T^{(d)}$ has exactly $m_d$ vertices of eccentricity $i$, and so $m_i \geq m_d$ and 
$|W^{(i)}|= m_i-m_d$. Note that $m_d \leq 3$, for $m_i - m_d = |W^{(i)}| \geq |U^{(i)}| \geq m_i - 3$. \\[1mm]
{\sc Case 3a:} $m_d=3$. \\ 
Then $T^{(d)} = T_{r,r,r}$. If $T=T^{(d)}$, then $T \in \mathcal{F}$, so assume 
that $T \neq T^{(d)}$. Then there exists $i \in \{r+1,r+2,\ldots,d-1\}$ for which 
$W^{(i)}\neq \emptyset$, and so  $|U^{(i)}| < |W^{(i)}| =m_i-3$, a contradiction
to \eqref{eq:eccentric-fixing-proof-2}. \\[1mm]
{\sc Case 3b:} $m_d =2$. \\
Then $T^{(d)}$ is a path, which we denote by $P$. 
We may assume that $T$ is not a path, and so $W \neq \emptyset$. 

The set $\bigcup_{i=r+1}^{d-1} U^{(i)}$ is not a fixing set of $T$, otherwise 
$F(T) \leq |\bigcup_{i=r+1}^{d-1} U^{(i)}| = \sum_{i=r+1}^{d-1} \max\{ m_i-3, 0\}$ and 
\eqref{eq:eccentric-fixing-proof-1} would be strict, contradicting our assumption. 
Hence there exists a nontrivial automorphism $\phi$ of $T$ that fixes $\bigcup_{i=r+1}^{d-1} U^{(i)}$.
Then $\phi$ fixes $W$ pointwise, but not $P$. If $v$ is a vertex of $P$ adjacent to 
some vertex of $W$, then $\phi$ fixes also $v$. 
It is easy to see that this implies that $v$ is the unique center vertex of $P$, that $P$ 
has even length $2r$, and that no other vertex of $P$ is adjacent to a vertex of $W$. 
Hence $W\cup \{v\}$ induces a subtree in $T$, which we denote by $T_W$. 

Let $k := {\rm ecc}_{T_W}(v)$. 
If $k=1$, then $T_W$ is a star with center $v$, and $T \in {\mathcal F}$, so we may assume
that $k>1$. We claim that 
\begin{equation} \label{eq:eccentric-fixing-proof-3}
|W^{(j)}|=1 \  \textrm{for} \  j=1,2,\ldots,k-1. 
\end{equation} 
Indeed, if 
$|W^{(j)}|\geq 2$ for some $j \in \{1,2,\ldots,k-1\}$, then some vertex 
$w_j \in W^{(j)}$ has a neighbour in $W^{(j+1)}$. 
We may assume that $U^{(j)}$ has been chosen to contain $w_j$. Now any automorphism of 
$T$ that fixes $W^{(j+1)}$, also fixes $w_j$, and so any automorphism that fixes 
$\bigcup_{i=r+1}^{d-1} U^{(i)} -\{w_i\}$ also fixes $W$ pointwise. Since any vertex
$u \in V(P)-\{w\}$ fixes $P$, the set  $ \{u\} \cup \bigcup_{i=r+1}^{d-1} U^{(i)} -\{w_i\}$
is a fixing set of $T$ of cardinality  $\sum_{i=r+1}^{d-1} \max\{ m_i-3, 0\}$, contradicting 
our assumption that  \eqref{eq:eccentric-fixing-proof-1} holds with equality. 
This proves \eqref{eq:eccentric-fixing-proof-3}.

It is easy to see that \eqref{eq:eccentric-fixing-proof-3} implies that $T_W$ is a broom with end 
$v$, and that the diameter of $T_W$ is at most $k$. Since $k \leq r$, we conclude that
 $T \in {\mathcal F}$, as desired. 
\end{proof}

For sharpness of the bounds in Theorem \ref{theo:fixing-distinguishing-vs-eccentric-sequence} consider
the following tree $T_X$, which was introduced and shown to be extremal among trees of a given eccentric 
sequence for various graph parameters in \cite{Dankelmann2020-eccentric, Dankelmann2021-eccentric}. 

Let $X$ be a sequence realizable as the eccentric sequence of a tree, 
$X=(r^{(m_r)}, (r+1)^{(m_{r+1})},\ldots,d^{(m_d)})$. By Theorem \ref{theo:lesniak} we have 
$r = \lceil \frac{d}{2} \rceil$, $m_r \in \{1,2\}$ and $m_i \geq 2$ for $i\in \{r+1, r+2,\ldots,d\}$. 
We define the tree $T_X$ as the tree obtained from
the path $v_0, v_1,\ldots,v_d$ by attaching $m_i-2$ pendant vertices to vertex $v_{i-1}$ for
$i = r+1, r+2,\ldots,d-1$. It is easy to verify that $X$ is the eccentric sequence of $T_X$.

\begin{figure}
    \centering
    \begin{tikzpicture}
    [semithick, every node/.style={circle, semithick, draw=black!100, fill=none, minimum size=5pt, inner sep=0pt}]

    \foreach \i in {0,...,6}
    {
    \node (v\i) [label={above:$v_{\i}$}] at (\i, 0) {};
    }

    \node (v3a) at (3,-1) {};
    \node (v4a) at (3.66, -1) {};
    \node (v4b) at (4, -1) {};
    \node (v4c) at (4.33, -1) {};
    \node (v5a) at (4.66, -1) {};
    \node (v5b) at (5.33, -1) {};

    \foreach \i in {3,4,5}
    {
    \draw (v\i) -- (v\i a);
    }

    \foreach \i in {4,5}
    {
    \draw (v\i) -- (v\i b);
    }

    \draw (v4) -- (v4c);

    \foreach[evaluate=\i as \myi using int(\i + 1)] \i in {0,...,5}
    {
    \draw (v\i) -- (v\myi);
    }
    
    \end{tikzpicture}
    \caption{A tree with maximum distinguishing number over all those with eccentric sequence $\bigl(3^{(1)}, 4^{(3)}, 5^{(5)}, 6^{(4)} \bigr)$}
    \label{fig:maxD(T)}
\end{figure}
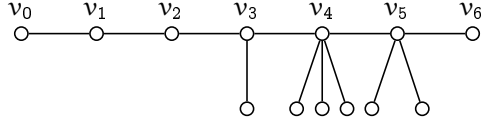

We claim that $T_X$ attains the bound in Theorem \ref{theo:fixing-distinguishing-vs-eccentric-sequence}.
Indeed, for $i=r,r+1,\ldots,d-2$, vertex $v_i$ is adjacent to $m_{i+1}-2$ leaves, and vertex $v_{d-1}$ is adjacent to $m_d-1$ leaves. Since in a distinguishing coloring the leaves adjacent to a vertex
need to receive different colors, it follows that any distinguishing coloring of $T_X$ 
has at least $\max\{m_d-1, \max_{r \leq i <d}(m_i-2) \} = M(T_X)$ colors, and so 
$D(T_X) = M(T_X)$.

We conclude this section by presenting three propositions that apply when more information is known regarding the eccentric sequence: a lower bound on the fixing number, a sufficient condition for a tree to have nontrivial symmetries (and hence nonzero fixing number), and an example of a tree meeting this sufficient condition which has fixing number $1$.

\begin{prop}\label{prop:bigmiddle}
Let $T$ be a tree with diameter $d$, radius $r$, and eccentric sequence $\bigl(r^{(m_r)}, (r+1)^{(m_{r+1})},\ldots,d^{(m_d)}\bigr)$.
If $m_i > m_{i-1} + m_{i + 1}$ for some $2 \leq i \leq d$ (where $m_{d+1} = 0$), then $F(T) \geq m_{i} - m_{i-1} - m_{i+1}$.
\end{prop}

\begin{proof}
    In order to minimize the fixing number of a tree with a given eccentric sequence, we must maximize the number of unique branches. Without considering the rest of the eccentric sequence we can consider ``leaf" automorphisms.
    
    Consider the $m_{i-1}$ vertices with eccentricity $i-1$. Call this set of vertices $K$. The $m_i$ vertices with eccentricty $i$ must be adjacent to these $m_{i-1}$. Call this set of vertices $L$. Additionally, the $m_{i+1}$ vertices with eccentricty $i+1$ must be adjacent to these $m_i$. Call this set of vertices $M$. See Figure~\ref{figure:KLM}.
    
    Notice at least $m_i-m_{i-1}$ vertices in $L$ must be adjacent to a vertex in $K$ that is adjacent to more than one vertex in $L$. To distinguish these $m_i-m_{i-1}$ vertices in $L$ from the other $m_{i-1}$, each must be adjacent to a vertex in $M$. However, $m_i-m_{i-1}-m_{i+1}$ of these vertices cannot be made adjacent to a vertex in $M$. Thus, there are $m_i-m_{i-1}-m_{i+1}$ vertices with eccentricity $i$ that must be fixed.
    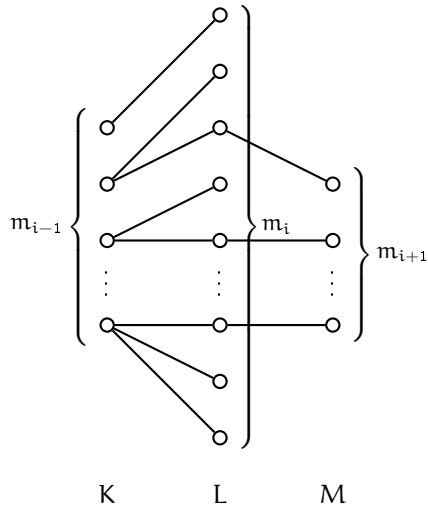
\begin{figure}[!h]
    \centering

    \begin{tikzpicture}
    [scale=.75, every node/.style={circle, thick, draw=black!100, fill=none, inner sep=0pt, minimum size=5pt}]

    \foreach \i in {0,1,2}
    {
    \node (k\i) at (0,\i) {};
    }
    \node[draw=none] (kdots) at (0,-.625) {$\vdots$};
    \node (k-2) at (0,-1.5) {};
    
    \foreach \i in {0,1,2,3,4}
    {
    \node (l\i) at (2,\i) {};
    }
    \node[draw=none] (ldots) at (2,-.625) {$\vdots$};
    \foreach \i in {-4,-3,-2}
    {
    \node (l\i) at (2,\i+.5) {};
    }
    
    \foreach \i in {0,1}
    {
    \node (m\i) at (4,\i) {};
    }
    \node[draw=none] (mdots) at (4,-.625) {$\vdots$};
    \node (m-2) at (4,-1.5) {};

    \node[rectangle, draw=none] (kbrace) at (-.25,0.25) {$\begin{cases}
        & \\[2.625cm]
    \end{cases}$};
    \node[draw=none] (k) at (-1.25,0.25) {\small $m_{i-1}$};

    \node[rectangle, draw=none] (lbrace) at (2.25,0.25) {$\begin{rcases}
        & \\[5.625cm]
    \end{rcases}$};
    \node[draw=none] (l) at (3,0.25) {\small $m_i$};

    \node[rectangle, draw=none] (mbrace) at (4.25,-.25) {$\begin{rcases}
        & \\[1.875cm]
    \end{rcases}$};
    \node[draw=none] (m) at (5.25, -.25) {\small $m_{i+1}$};

    \draw[thick] (k2) -- (l4);
    \draw[thick] (l3) -- (k1) -- (l2) -- (m1);
    \draw[thick] (l1) -- (k0) -- (l0) -- (m0);
    \draw[thick] (m-2) -- (l-2) -- (k-2) -- (l-3);
    \draw[thick] (k-2) -- (l-4);

    \node[draw=none] (K) at (0,-4.5) {$K$};
    \node[draw=none] (L) at (2,-4.5) {$L$};
    \node[draw=none] (M) at (4,-4.5) {$M$};

    \end{tikzpicture}
    \caption{Example of edges between $K$, $L$, and $M$ (see proof of Proposition~\ref{prop:bigmiddle})}
    \label{figure:KLM}
    \end{figure}
\end{proof}

A simple case of the previous proposition is that, if $m_d > m_{d-1}$ for a tree $T$ with eccentric sequence $\bigl(r^{(m_r)}, (r+1)^{(m_{r+1})},\ldots,d^{(m_d)}\bigr)$, then $F(T) \geq m_d - m_{d-1} \geq 1$.
In particular, $T$ has some nontrivial symmetry.
This fact generalizes as follows.

\begin{prop}\label{prop:eventuallymonotoneincreasing}
    Let $T$ be a tree with diameter $d$, radius $r$, and eccentric sequence $\bigl(r^{(m_r)}, (r+1)^{(m_{r+1})},\ldots,d^{(m_d)}\bigr)$.
    If $m_s < m_{s+1} = \cdots = m_d$ for some $s$, then $T$ is not asymmetric.
   
\end{prop}

\begin{proof}
    
    It is not hard to show that every vertex of eccentricity $i > r$ in $T$ has exactly one neighbor of eccentricity $i - 1$.
    If any vertex of eccentricity $d - 1$ has two neighbors of eccentricity $d$, then at least one must be contained in every fixing set for $T$, and we are done.
    Otherwise, every vertex of eccentricity $d-1$ has degree $2$: it has a single neighbor of eccentricity $d - 2$, and a single leaf neighbor of eccentricity $d$.
    Note that, since $m_s < m_d$, there is some (largest) $i \geq s$ for which a vertex of eccentricity $i$ has degree at least $3$.
    Such a vertex $v$ has $\deg(v) - 1$ neighbors of eccentricity $i + 1$, and each of these lies on a path from $v$ to a vertex of eccentricity $d$.
    All but at most one of these paths must contain a vertex in every fixing set for $T$.
    Since $\deg(v) - 2 \geq 1$, the proof is complete.
   
\end{proof}

\begin{prop}
    For any integers $k$ and $r$ such that $2 \leq k \leq r + 1$, there exists a tree with eccentric sequence $\bigl(r, (r+1)^{(k)}, \ldots, d^{(k)} \bigr)$ and fixing number $1$.
\end{prop}

\begin{proof}
    We begin with a path $P$ of length $2r$, $P = v_0v_2 \cdots v_{2r}$.
    If $k = 2$, then we're done.
    If $2 < k \leq r + 1$, we construct $T$ as follows.
   
    To each vertex $v_j$ in $P$, $1 \leq j \leq k-2$, we append a pendent path of length $j$.
    To $v_r$, we append $k - 2$ pendent paths of lengths $r - 1, \ldots, r - k + 2$.
    It is not hard to check that there are $k$ vertices of eccentricity $r+1, \ldots, 2r$.
    Further, any two vertices with the same eccentricity either lie on pendent paths from $v_r$ of differing lengths, and thus are fixed by any automorphism, or one of them is either a $v_i$ ($0 \leq i \leq r$) or lies on a pendent path from a $v_i$ ($1 \leq i < r$).
    The only such case which allows for a nontrivial automorphism is when these two vertices are the ones of eccentricity $2r$ whose common neighbor is $v_1$.
    Fixing one of these vertices fixes all of $T$.
\end{proof}


\section{Future Work}\label{sec: future work}

In Section~\ref{sec:fixing densities}, we showed that when $T$ is a 2-distinguishable tree of order $n \geq 3$, it has fixing density at most 4/11, and if $T$ is a $D$-distinguishable tree, $D \geq 3$, then $T$ has fixing density at most $\frac{D-1}{D+1}$. These bounds do not hold for $D$-distinguishable graphs in general.
Any disjoint union of complete graphs $K_D$, for instance, has fixing density $1 - 1/D$.
It is not hard to see that this is the maximum fixing density of a $D$-distinguishable graph, for the union of any $D-1$ color classes in a distinguishing $D$-coloring comprises a fixing set.
On the other hand, it is not clear what the maximum fixing density of a connected $D$-distinguishable graph of order $n > D$ should be.
We now exhibit a family of $2$-distinguishable graphs of order $n$ which have fixing numbers greater than $\frac{1}{2} ( n - \log_2{n})$, which can be generalized to a family of $D$-distinguishable graphs with fixing numbers at least $\frac{D-1}{D}(n - \log_2{n} + \log_2{D})$.

Consider the following graph $G_k$ for $k\geq 7$.
Let $G_k$ be obtained from an asymmetric graph $H_k$ on $k$ vertices by adding, 
for each nonempty subset $U$ 
of $V(H)$, two new vertices $v_U$ and $w_U$, and adding all edges between $\{v_U,w_U\}$ and $U$. 
Then $G_k$ is connected and has $k + 2^{k+1}-2$ vertices. 
Every automorphism of $G_k$ fixes $V(H_k)$ since these are the only vertices of degree greater than
$k+1$, and since $H_k$ is asymmetric, every automorphism of $G_k$ fixes $H_k$ pointwise. 
Hence coloring the vertices of $H_k$ and all $v_U$ with color $1$, and all $w_U$ with color $2$
is a distinguishing coloring, so $G_k$ is $2$-distinguishable. 
Since for every nonempty subset $U$ of $V(H_k)$ there exists an automorphism that swaps $v_U$ and 
$w_U$ and leaves all other vertices fixed, every fixing set contains at least $2^k-1$ vertices,
and since $H_k$ is asymmetric, this is optimal. Hence 
$F(G_k) =  2^k-1 = \frac{n}{2} - O(\log n)$. 

We can generalize this construction for $D$-distinguishable graphs, $D \geq 3$, by adding new vertices $v_{1_U}, v_{2_U}, \dots, v_{D_U}$ to an asymmetric $H_k$ in a similar fashion. With this construction, every fixing set contains at least $(D-1)(2^k - 1)$ vertices when $k > D$ and hence, we have fixing number greater than or equal to $\frac{D-1}{D}(n - \log_2{n} + \log_2{D})$, which again approaches the natural upper bound of $(D-1)n/D$.

\begin{ques}
    Are there connected $D$-distinguishable graphs of large order $n$ and fixing number at least $\frac{D-1}{D}(n - o(\log{n}))$?
\end{ques}

In Section~\ref{sec:universal trees} we characterized the $D$-distinguishable trees with radius $r$, $r \geq 1$, $D \geq 2$ by constructing a universal tree $T_r^D$ such that all $D$-distinguishable trees of radius at most $r$ can be found as branched subgraphs of $T_r^D$. We also found a similar collection of universal trees $U_r^D$, $r \geq 1$, $D \geq 2$ for all $D$-distinguishable trees $T$ with the property that $\rho^D(T) = F(T)$. 

Let $\mathcal{T}_n$ be the family of $n$-vertex trees. In~\cite{chung1978graphs}, Chung and Graham considered the minimum number of edges $s(n)$ in an $n$-vertex $\mathcal{T}_n$-universal graph. Since then significant work has been done to determine lower and upper bounds on this parameter (see~\cite{becker2026improved, kaul2025universal}, for example, which also generalize their bounds to graphs with treewidth $k$).
While our universal tree $T_r^D$ is clearly of minimum size when the radius $r$ is fixed, it is not obvious what the minimum size of a universal tree should be when the order, rather than the radius, of the $D$-distinguishable trees it should contain is fixed.

\begin{ques}
    Let $T_{D_n}$ be a universal tree which contains all $D$-distinguishable trees of order $n$, $D \geq 1$. What is the minimum number of edges in $T_{D_n}$?
\end{ques}

Lastly, in a different direction, the penultimate section of our paper leaves some intriguing questions for future research relating eccentric sequences and symmetries in trees.
One natural direction would be to characterize the eccentric sequences of asymmetric trees.
By Proposition~\ref{prop:eventuallymonotoneincreasing}, the eccentric sequence of an asymmetric tree cannot be eventually increasing.
On the other hand, we can find trees with nontrivial symmetries and eccentric sequences which are not eventually increasing using Proposition~\ref{prop:bigmiddle}; consider, for example, the tree $S_{2,2,1,1}$ with eccentric sequence $\bigl( 2, 3^{(4)}, 4^{(2)} \bigr)$.

\begin{ques}
    For which eccentric sequences does there exist a representative asymmetric tree?
\end{ques}

\section{Acknowledgements}

This work is supported by the National Science Foundation under NSF Award 2015425.
Any opinions, findings, and conclusions or recommendations expressed in this material are those of the authors and do not necessarily reflect those of the National Science Foundation.

We thank the organizers and founders of the SAMSA-Masamu program, without whom
this research would not have been possible. 

Meaningful work on this project was completed while the fifth author was visiting the Discrete Mathematics Group at the Institute for Basic Science in Daejeon, South Korea and the fifth author thanks IBS for the generous hospitality.

\bibliographystyle{plain}
\bibliography{references}

\end{document}